\documentclass[12pt]{amsart}
\usepackage{color,graphicx,array, amssymb, amscd}
\newtheorem{Theorem}{Theorem}[section]
\newtheorem{Lemma}[Theorem]{Lemma}
\newtheorem{Proposition}[Theorem]{Proposition}
\newtheorem{Corollary}[Theorem]{Corollary}
\theoremstyle{definition}
\newtheorem{Definition}{Definition}[section]
\theoremstyle{Remark}
\newtheorem{Example}{Example}
\newtheorem{Remark}[Theorem]{Remark} 
\newtheorem{Problem}[Theorem]{Problem} 

\numberwithin{equation}{section}
\def\Vec#1{\mbox{\boldmath $#1$}}
\newcommand{\ad}{\operatorname{Ad}}

\newcommand{\nm}{\nabla^{\mu}}
\newcommand{\nt}{{}^{(t)} \nabla}
\newcommand{\can}{{}^{(-1)}\nabla}
\newcommand{\anti}{{}^{(1)}\nabla}
\newcommand{\neutral}{{}^{(0)}\nabla}

\newcommand{\sk}{\operatorname {skew}}
\newcommand{\sym}{\operatorname {sym}}

\begin{document}
\title{Homogeneous statistical manifolds}

\author[J.~Inoguchi]{Jun-ichi Inoguchi}
\address{Department of Mathematics, Hokkaido University, 
Sapporo, 060-0810, Japan}
\email{inoguchi@math.sci.hokudai.ac.jp}
\thanks{The first named author is partially supported by 
Kakenhi JP19K03461, JP23K03081}
%
\author[Y.~Ohno]{Yu Ohno}
\address{Department of Mathematics, Hokkaido University, 
Sapporo, 060-0810, Japan}
\email{ono.yu.h4@elms.hokudai.ac.jp}
\thanks{The second named author is supported by JST SPRING, Grant Number JPMJSP2119.}

 \dedicatory{Dedicated to professor Takashi Kurose on the occasion of his 60th birthday}
\pagestyle{plain}

\begin{abstract} 
The methods of Information geometry have been glowing up to develop various subjects of theoretical physics, 
including quantum information systems. The present article has two purposes. 
The first one is to develop general theory of homogeneous statistical manifolds. 
In particular we construct explicit examples of homogeneous statistical manifolds 
of low dimension. The second purpose is to classify $3$-dimensional Lie groups 
admitting non-trivial conjugate symmetric left invariant statistical structure. 
\end{abstract}
\maketitle
\section*{Introduction}

According to Amari \cite{Amari}, 
``Information geometry is a method of exploring the world of information by means 
of modern geometry". Information geometry provides a differential 
geometric treatment of statistical models, 
especially, manifolds 
of probability distributions.

From differential geometric point of view, 
the notion of statistical manifold is the prominent clue 
of information geometry. 
A statistical manifold is defined as a triplet $(M,g,\nabla)$ consisting 
of a smooth manifold $M$, a Riemannian metric $g$ on $M$ and a 
linear connection $\nabla$ on $M$ compatible to $g$. A pair $(g,\nabla)$ 
consisting a Riemannian metric $g$ together with a compatible linear connection is 
referred as to a \emph{statistical structure}.
 
The Riemannian metrics of statistical manifolds have its origin 
in the fundamental work by Rao \cite{Rao} in 1945. Rao used Fischer information 
to define a Riemannian metric on a space of probability distributions. 
Efron\cite{Efron} extended Rao's idea to the higher-order 
asymptotic theory of statistical inference. Amari and Chentsov 
introduced linear connections on statistical models which are nowadays 
called \emph{Amari-Chentsov connections} \cite{AN}, independently.

Information geometry has been generalized for quantum information systems. 
For instance, Fujiwara \cite{Fujiwara} proved that 
every monotone metric on a two-level 
quantum state space admits a Hessian structure, 
\textit{i.e.}, flat statistical structure 
(see also \cite{CMH}).

Beyonds, statistics and differential geometry, 
the frame work of information geometry has been 
glowing up to develop various subjects of 
theoretical physics. Here we mention 
few examples. 
Matsueda and Suzuki \cite{MS} 
examined the Ba{\~n}ados–Teitelboim–Zanelli (BTZ) black 
hole in terms of the information geometry and 
considered what kind of quantum information produces 
the black hole metric in close connection with AdS/CFT correspondence. 
Ito and Ohga \cite{Ito,OhIt21,OhIt22} constructed 
an information-geometric structure for chemical thermodynamics. 
Tsuchiya and Yamashiro \cite{TY} 
discussed information metrics in field theory via gauge/gravity correspondence.
They consider a quantum information metric that measures 
the distance between the ground states of a CFT and a 
theory obtained by perturbing the CFT. 
They found a universal formula that represents the quantum 
information metric in terms of back reaction to the AdS bulk geometry. Recently, 
\cite{IP}, 
Iosifidis and 
Pallikaris 
proposed a bi-connection theory 
of gravity.

As is well known, in cosmology and general relativity, 
spacetimes are regarded as time-oriented Lorentzian $4$-manifolds, 
\textit{i.e.}, semi-Riemanninan 4-manifolds of signature $(-,+,+,+)$. 
Thus general relativity has been investigated by virtue of Lorenzian geometry. 
To construct specific spacetimes, especially 
solutions to Einstein's gravity equations, 
gluing of homogeneous spacetimes is a very effective technique. 
On this reason, homogeneous spacetimes have been paid much attention of 
researchers of general relativity as well as differential geometers. 
On the other hand, in information geometry, homogeneous statistical manifolds 
have not been well studied, yet. The reason is the lack 
of explicit examples. 
To provide particularly nice classes of statistical manifolds, 
we need ample explicit examples 
of homogeneous statistical manifolds.

The present article has two purposes. 
The first one is to develop general theory of 
homogeneous statistical manifolds. 
After establishing prerequisite knowledge on 
information geometry in Section \ref{sec:1}, we start our 
investigation of homogeneous statistical manifolds from Section 
\ref{sec:2}. A statistical manifold $(M,g,\nabla)$ is said to be a 
\emph{homogeneous statistical manifold} if there exits a 
Lie group $G$ acting transitively on $M$ so that the action is 
isometric with respect to $g$ and affine with respect to $\nabla$. As a 
result $M$ is expressed as a coset manifold 
$G/H$ where $H$ is an isotropy subgroup at some 
point.

In Riemannian geometry, 
Ambrose and Singer \cite{AS} proved that 
every homogeneous Riemannian manifold 
$M=G/K$ admits a certain tensor field $S$ called a
\emph{homogeneous Riemannian structure}. Conversely 
a Riemannian manifold $(M,g)$ together with a 
homogeneous Riemannian structure $S$ is locally 
homogeneous. By virtue of Ambrose-Singer theorem, 
our first result (Proposition \ref{prop:2.2.1}) 
says that a locally homogeneous 
statistical manifold is  characterized as  a Riemannian manifold
$(M,g)$ equipped with a homogeneous Riemannian structure $S$ and 
a cubic form $C$ which is parallel with respect the Ambrose-Siger connection 
$\tilde{\nabla}=\nabla^{g}+S$, where $\nabla^g$ is the Levi-Civita connection 
of the metric $g$.

As an application of Proposition \ref{prop:2.2.1}, in the next subsection 
\ref{sec:2.3}, we construct explicit examples of homogeneous Sasakian statistical manifolds. 
Roughly speaking, Sasakian manifolds are regarded as odd-dimensional counterparts of 
K{\"a}hler manifolds. In this decade, Sasakian manifolds, especially Sasaki-Einstein manifolds 
have been paid much attention by theoretical physicists. Because
the AdS/CFT correspondences is interpreted as a conjecture that the type IIB 
string theory on the product manifold $\mathrm{AdS}_5\times\mathrm{SE}_5$
of $1+4$-dimensional anti de Sitter spacetime 
$\mathrm{AdS}_5$ and Sasaki-Einstein $5$-manifold 
$\mathrm{SE}_5$ is equivalent to 4D N=1 superconformal quiver gauge theory. 
Sasakian manifolds have another origin on physics. 
Sasakian manifolds are contact manifolds equipped with special Riemannian metrics.
Theory of contact manifold has 
its origin in analytical mechanics and optics (Huygens' principle, wave fronts). 
Moreover in Thermodynamics, contact manifolds 
play a fundamental role (see \cite{HL}).

From Section \ref{sec:3} we concentrate our attention 
to homogeneous statistical manifolds with 
\emph{trivial isotropy}. Under this assumption, 
the homogeneous statistical manifold $M$ has the form 
form $M=G/\{e\}$, where $e$ is the 
identity element of the Lie group $G$. 
Moreover, $M$ is identified with the Lie group $G$ 
and the statistical structure $(g,\nabla)$ is a left invariant one on $G$.
A Lie group $G$ equipped with a left invariant 
statistical structure is called a \emph{statistical Lie group} \cite{FIK}. 
We prepare ingredients for the theory of left invariant linear connection and 
Riemannian metrics in Section \ref{sec:3} and 
Section \ref{sec:4}, respectively. Combining the results in these two sections, 
we develop a  general theory of left invariant statistical structure 
on Lie groups in Section \ref{sec:5}.
In Proposition \ref{prop:5.1} we give a general method to construct 
left invariant statistical structures on Lie groups from the 
infinitesimal data on the Lie algebra. 

It should be remarked that statistical structure is 
sometimes called a Codazzi structure. 
In Riemannian geometry, a Codazzi tensor is a 
symmetric $(0,2)$-tensor $h$ which satisfies (see Section \ref{sec:1.2}).
\[
(\nabla^{g}_{X}h)(Y,Z)=(\nabla^{g}_{Y}h)(X,Z).
\]
Obviously on a Riemannian manifold $(M,g,\nabla^{g})$ 
equipped with another Riemannian metric $h$ which is Codazzi, then 
$(M,h,\nabla^g)$ is a statistical manifold.

The notion of Codazzi tensor is closely related to Einstein metrics. Indeed, 
the Ricci tensor field $\mathrm{Ric}^g$ of an Einstein manifold 
$(M,g)$ is a Codazzi tensor field. More generally, 
the Ricci tensor field of a Riemannian manifold 
is a Codazzi tensor field if and only if the Riemannian curvature 
is a Yang-Mills gauge fields on the tangent bundle. 
On this reason, a Riemannian manifold $(M,g)$ is said to be a 
\emph{space of harmonic curvature} if its Ricci tensor field is 
Codazzi. Clearly the notion of space of harmonic curvature is a 
generalization of that of Einstein manifold. In Section \ref{sec:5.2}, 
we discuss left invariant Codazzi tensor fields on Lie groups.

Section \ref{sec:6} to Section \ref{sec:8} are devoted to the second purpose 
of this article, which is to construct explicit examples 
of statistical Lie groups in low-dimension. 
First we investigate statistical Lie groups of dimension $2$. 
In our previous work \cite{FIK}, we showed that the statistical manifold 
of normal distribution carries a natural statistical Lie group structure. 
Moreover we gave a purely differential 
geometric characterization of Amari-Chentsov $\alpha$-connections on the 
statistical manifold of the normal distribution. More precisely 
we proved that the only left invariant connection on the 
statistical Lie group of normal distributions compatible 
to the Fischer metric which satisfies the conjugate symmetry. 
Thus Amari-Chentsov $\alpha$-connections are characterized by 
a purely differential geometric property ``conjugate-symmetry" 
(see Section \ref{sec:1.8}). 

It should be emphasized that on a statistical manifold $(M,g,\nabla)$ we can not 
introduce the notion of sectional curvature as the manner in Riemannian geometry 
because $\nabla$ is non-metrical in general. Under the conjugate symmetry, one can introduce 
the notion of sectional curvature (\emph{statistical sectional curvature}). 

Next we will turn our attention to 
statistical Lie group of $t$-distributions (Section \ref{sec:6.4}).
The Student's $t$-distribution is a subfamily of 
$q$-normal distributions. 
The Boltzmann-Gibbs probability distribution, seen as a statistical model, belongs to the 
so-called exponential family. 
Motivated by Boltmann-Gibbs probability distribution, 
$q$-exponential function $\exp_q$ was introduced in \emph{Tsallis statistics}. 
Replacing the usual exponential function by $\exp_q$ in the definition
of statistical model of exponential families, one obtain the statistical
model of $q$-exponential family. Naudts \cite{Naudts} discussed 
the role of $q$-exponential families in physics. 
The distribution of the momentum of a single particle is a 
$q$-normal distribution. He discussed definition of the temperature of small isolated systems. 
In Section \ref{sec:6.4} we describe statistical Lie group structure for 
$t$-distributions. We expect statistical Lie group structure of $t$-distributions 
will be applied for statistical physics.

The Section \ref{sec:7} and \ref{sec:8} will be devoted for 
statistical Lie groups of dimension three. 
One may  expect that conjugate symmetric homogeneous 
statistical manifolds constitute a particularly nice 
class of statistical manifolds. 
In these two sections, we shall look for $3$-dimensional examples of 
conjugate symmetric statistical Lie groups. 
According to Milnor \cite{Milnor}, 
the class of 
$3$-dimensional Lie groups with left invariant Riemannian metric are 
divided into two classes; unimodular ones (Section \ref{sec:7}) and 
non-unimodular ones (Section \ref{sec:8}, see Milnor \cite{Milnor}). 
In Section \ref{sec:7}, we classify 
$3$-dimensional unimodular Lie groups 
which admit left invariant conjugate symmetric statistical 
structure. We prove that the only simply 
connected $3$-dimensional unimodular Lie groups
which admit \emph{non-trivial} left invariant conjugate symmetric statistical 
structure are Euclidean $3$-space $\mathbb{E}^3$, 
the universal covering of the Euclidean motion group $\widetilde{\mathrm{E}_2}$
and the special unitary group $\mathrm{SU}_2$
(Theorem \ref{thm:uni}).

In Section \ref{sec:8}, we prove that the only simply 
connected $3$-dimensional non-unimodular Lie groups
which admit \emph{non-trivial} left invariant conjugate symmetric statistical 
structure are the solvable Lie group model 
of the hyperbolic $3$-space and 
the product Lie group $\mathbb{H}^2\times\mathbb{R}$ of 
the solvable Lie group model of the hyperbolic plane and 
the real line (Theorem \ref{thm:nonuni}). 

This article will ends with proposing some future problems. 

\section{Statistical manifolds}\label{sec:1}
\subsection{Statistical structures}\label{sec:1.1}
Let $M$ be a manifold,
$g$ a Riemannian
metric and 
$\nabla$ a torsion free linear connection.
Then $\nabla$ is said to be 
\textit{compatible} to $g$ if
the covariant derivative 
$C:=\nabla g$ is symmetric. 

More explicitly $\nabla g$ satisfies
\[
(\nabla_{X}g)(Y,Z)=(\nabla_{Y}g)(X,Z)
\]
for all $X$, $Y$, $Z\in\mathfrak{X}(M)$. Here 
$\mathfrak{X}(M)$ denotes the Lie algebra 
of all smooth vector fields on $M$.

\begin{Definition}\label{def:kurose}
{\rm 
A pair $(g,\nabla)$ of a Riemannian metric 
and a compatible linear connection is
called a \emph{statistical structure}
on $M$. A manifold $M$ together with a statistical structure
is called a \emph{statistical manifold}. 
In particular a statistical manifold  $(M,g,\nabla)$
is traditionally called a \emph{Hessian manifold} 
if $\nabla$ is a flat connection \cite{Shima}.
}
\end{Definition}

A Riemannian manifold $(M,g)$ together
with Levi-Civita connection
$\nabla^{g}$ of $g$ is a typical
example of statistical manifold.
In other words, statistical
manifolds
can be regarded as generalizations
of Riemannian manifolds.

Let $TM$ and $T^{*}M$ be the tangent bundle and 
cotangent bundle of the a statistical manifold $(M,g,\nabla)$. 
Then the Riemannian metric $g$ 
induces a bundle isomorphism 
$\flat:TM \to T^{*}M$ by:
\[
(X^{\flat})Y=g(X,Y),\ \  X,Y 
\in \mathfrak{X}(M).
\]
Next the affine connection $\nabla$ induces a 
connection $\nabla^*$ on $T^{*}M$ by 
\[
(\nabla^*_{X}\omega)Y:=X(\omega(Y))-\omega(\nabla_{X}Y),\ \
\omega \in \mathfrak{X}^{*}(M),\ X,\ Y \in \mathfrak{X}(M).
\] 
The connection $\nabla^*$ is called the \emph{dual connection}
of $\nabla$. 
 
On the other hand, the 
\emph{conjugate connection} ${}^\dag\nabla$ of
$\nabla$ with respect to $g$ is 
introduced by the following formula:
\begin{equation}\label{eq:conjugate}
X \cdot 
g(Y,Z)=
g(\nabla_X Y,Z)+g(Y,{}^{\dag}\nabla_{X}Z),
\ X,Y,Z \in \mathfrak{X}(M).
\end{equation}
Obviously $\nabla={}^\dag\nabla$ if and only if $\nabla$ coincides with the 
Levi-Civita connection $\nabla^g$. 

Moreover we get $\flat \circ {}^{\dag}\nabla=\nabla^{*}\circ \flat$.
Namely, if we identify $TM$ with $T^{*}M$ via $\flat$, then 
${}^\dag\nabla$ is identified with $\nabla^*$.
On this reason, hereafter we denote the conjugate connection by 
$\nabla^*$ and call it \emph{dual connection}.

In 1990, Kurose \cite{Kurose} formulated the 
notion of statistical manifold as a 
semi-Riemannian manifold 
$(M,g)$ equipped with a pair $(\nabla,\nabla^{*})$ 
of linear connections satisfying 
the relation \eqref{eq:conjugate}. 
In this paper we restrict our attention to 
Riemannian metrics, \textit{i.e.}, 
positive 
definite semi-Riemannian metrics.

Define the tensor field $K$ of type 
$(1,2)$ by 
\[
g(K(X,Y),Z)=C(X,Y,Z)=(\nabla_{X}g)(Y,Z)
\]
for all $X$, $Y$, $Z$. 
For every vector field $X$, one 
can associate an endomorphism field $K(X)$ by
\[
K(X)Y=K(X,Y).
\]
Since $C$ is totally symmetric, $K(X)$ is 
self-adjoint with respect to $g$,
\textit{i.e.},
\[
g(K(X)Y,Z)=g(Y,K(X)Z)
\]
and satisfies
\[
K(X)Y=K(Y)X.
\]
We call this tensor field $K$ 
the \emph{skewness operator} of
$(N,g,\nabla)$.
The difference of
$\nabla$ and $\nabla^g$ is
given by
\[
\nabla-\nabla^g=-\frac{1}{2}K.
\] 
This formula implies that 
$\nabla^g$ is the ``mean" of $\nabla$ and 
$\nabla^{*}$:
\[
\nabla^g=\frac{1}{2}(\nabla+\nabla^{*}).
\]
As we have seen before, 
for every statistical manifold
$(M,g,\nabla)$, there exists a
naturally associated symmetric trilinear form $C$.

Conversely let $(M,g,C)$ be a Riemannian 
manifold with symmetric trilinear form 
$C$ (called a \emph{cubic form}). Then define the tensor field $K$ by
\[
g(K(X)Y,Z)=C(X,Y,Z).
\]
and an affine connection $\nabla$ by
$\nabla=\nabla^g-K/2$.
Then  $\nabla$ is of
torsion free and satisfies
$\nabla g=C$. 
Hence the triplet $(M,g,\nabla)$ becomes
a statistical manifold.

Thus to equip a statistical structure $(g,\nabla)$ 
is equivalent to equip a pair of structue
$(g,C)$ consisting of a 
Riemannian metric $g$ and a trilinear form $C$.
Lauritzen 
introduced the notion of statistical manifold as a Riemannian 
manifold $(M,g)$ together with a trilinear form $C$.

Nowadays, Definition \ref{def:kurose} due to Kurose \cite{Kurose1994} is the standard 
definition of statistical manifold.

\begin{Example}{Statistical models}
{\rm Let $M$ be a family of probability 
distributions. Assume that $M$ is 
represented as 
\[
M=\{p(x,\Vec{\xi})\>|\>\Vec{\xi}\in\Xi\},
\]
where 
$\Xi$ be a region 
of $\mathbb{R}^n$ with 
Cartesian coordinates 
$(\xi^1,\xi^2,\dots,\xi^n)$ (called the  
\emph{parametric space}).
In addition we assume that 
the correspondence 
$\Vec{\xi}\longmapsto p(\cdot,\Vec{\xi})$ is 
bijective. Then $M$ is said to be a 
\emph{statistical model} 
or \emph{parametric model} of the 
probability distribution 
determined by $p$. 

On an $n$-dimensional statistical model 
$M$, we set
\[
g_{ij}=E\left[
\frac{\partial}{\partial \xi^i}\log p(x,\Vec{\xi})\cdot
\frac{\partial}{\partial \xi^j}\log p(x,\Vec{\xi})
\right],
\]
where $E$ be the expectation 
determined by the pdf $p$. 
Here after we assume that 
\begin{enumerate}
\item $g_{ij}$ is finite and smooth.
\item the matrix valued function $(g_{ij})$ is positive definite.
\end{enumerate}
Then 
\[
g^{\mathsf F}=
\sum_{i,j=1}^{n}
g_{ij}\,
d\xi^{i}\,d\xi^{j},
\] 
is a Riemannian metric on $M$ called 
the \emph{Fisher metric} of $M$.

Let us introduce a one-parameter family of linear connections 
on $M$.
First we define 
\[
\varGamma_{ijk}^{(\alpha)}
:=E\left[
\left(
\frac{\partial^2\log p(x,\Vec{\xi})}{\partial \xi^i\xi^j}
+\frac{1-\alpha}{2}
\frac{\partial\log p(x,\Vec{\xi})}{\partial \xi^i}
\frac{\partial\log p(x,\Vec{\xi})}{\partial \xi^j}
\log p(x,\Vec{\xi})
\right)
\cdot
\left(
\frac{\partial\log p(x,\Vec{\xi})}{\partial \xi^k}
\right)
\right].
\]
 Next we set
\[
{}^{(\alpha)} \varGamma^{l}_{ij}:=\sum_{k=1}^{n} g^{lk}
\varGamma_{ijk}^{(\alpha)}.
\]
Then the family $\{{}^{(\alpha)} \varGamma^{l}_{ij}\}
_{1\leq i,j,k\leq n}$ defines a 
linear connection $\nabla^{(\alpha)}$.
One can see that the linear connection 
$\nabla^{(0)}$ coincides with 
the Levi-Civita connection $\nabla^{g^{\mathsf F}}$ of the 
Fischer metric. 
The linear connection $\nabla^{(1)}$ is simply denoted by 
$\nabla$ and called the \emph{exponential 
connection} (e-connection). The 
exponential connection is also denoted by
$\nabla^{ (\mathrm{e}) }$. 
On the other hand, 
$\nabla^{(-1)}$ is 
called the \emph{mixture 
connection} (m-connection). The 
mixture connection is also denoted by
$\nabla^{(\mathrm{m})}$.
One can see that 
$(M,g^{\mathsf F},\nabla)$ is a 
statistical manifold. Moreover the dual connection 
$\nabla^{*}$ of 
$\nabla=\nabla^{(\mathrm{e})}$ is 
$\nabla^{(\mathrm{m})}$. 
The linear connections 
$\nabla^{(\alpha)}$ are called the \emph{Amari-{\^ C}encov $\alpha$-connections} 
(Amari-Chentsov $\alpha$-connections).
}
\end{Example}

\begin{Example}[The exponential family]
{\rm A statistical model 
$M$ of the form
\[
M=\left\{
p(x,\Vec{\theta})\>
\biggr|
\>
 \exp\left[
C(x)+\sum_{i=1}^{n}\theta^{i}F_{i}(x)-\psi(\Vec{\theta})
 \right],\>
 \Vec{\theta}\in\Theta
 \right\}
 \]
 is called an \emph{exponential family}. 
 Here $C(x)$, $F_1(x),F_2(x),\dots,F_n(x)$ are smooth functions 
 of $x$ and $\psi$ is a smooth function 
 on $\Theta$. The coordines $\Vec{\theta}
 =(\theta^1,\theta^2,\dots,\theta^n)$ is 
 called a natural parameters. Note that 
 \[
 \psi(\Vec{\theta})=\log 
 \int \exp\left[C(x)+\sum_{i=1}^{n}\theta^{i}F_{i}(x)
 \right]\,
 dx.
 \]
 For instance, if we choose 
 \[
 C(x)=0,
 \quad 
 F_{1}(x)=x,\quad 
 F_{2}(x)=x^2,
 \quad 
 \theta^1=\frac{\mu}{\sigma^2},
 \quad 
 \theta^2=-\frac{1}{2\sigma^2},
 \]
 then 
 \[
 \psi(\theta^1,\theta^2)=
 -\frac{(\theta^1)^2}{4\theta^2}+
 \frac{1}{2}\log
 \left(-\frac{\pi}{\theta^2}\right).
 \]
 The probability distribution is 
 the normal distribution 
 $N(\mu,\sigma)$.
}
\end{Example}

\subsection{Codazzi tensor fields}\label{sec:1.2}

\begin{Definition}
{\rm Let $(M,g,\nabla^g)$ be a Riemannian manifold. 
A symmetric tensor field $h$ of type $(0,2)$ is said to be a 
\emph{Codazzi tensor field} if it satisfies 
\[
(\nabla^{g}_{X}h)(Y,Z)=(\nabla^{g}_{Y}h)(X,Z), 
\quad X,Y,Z\in\mathfrak{X}(M).
\]
}
\end{Definition}
According to this definition, we obtain.
\begin{Proposition}
Let $(M,g,\nabla^g)$ be a Riemannian manifold and $h$ a Riemannian metric on $M$. 
If $h$ is a Codazzi tensor field, then 
$(M,h,\nabla^g)$ is a statistical manifold.
\end{Proposition}
In some literature, statistical structure $(h,\nabla)$ is called a 
\emph{Codazzi structure} and statistical manifolds are called 
\emph{Codazzi manifolds}. However, the original notion of 
Codazzi tensor field, the connection 
is chosen as the Levi-Civita connection of the 
prescribed Riemannian metric.

Here we would like to mention about spaces of harmonic curvature 
in Riemannian geometry. The notion of space of harmonic curvature was introduced as 
a generalization of Einstein manifold. 
In addition spaces of harmonic curvature provide 
specific examples 
of Yang-Mills gauge-fields.

Let us assume that $(M,g,\nabla^g)$ be compact and oriented by the 
volume element $dv_{g}$. 
Then on a vector bundle $E$ over $M$ equipped with 
a fiber metric $g$ and a metric connection $D$. Then the 
\emph{Yang-Mills functional} over the affine space of 
metric connections of $E$ is defined by 
\[
\mathcal{Y}\!\mathcal{M}(D)=\int_{M}\frac{1}{2}|\!|R^{D}|\!|^2\,dv_{g},
\]
where $R^D$ is the curvature of $D$;
\[
R^{D}(X,Y)=D_{X}D_{Y}-D_{Y}D_{X}-D_{[X,Y]},
\quad X,Y\in\mathfrak{X}(M).
\]
Critical points of the Yang--Mills functional is 
called \emph{Yang-Mills connections}. 
It is known that the Levi-Civita connection $\nabla^g$ is a
Yang-Mills connection on the tangent bundle 
if and only if its Ricci tensor field 
is a Codazzi tensor field. 
Based on this fact, a Riemannian manifold 
$(M,g,\nabla^g)$ is said to be a 
\emph{space of harmonic curvature} if its 
Ricci tensor field is a Codazzi tensor field even if $M$ is non-compact. 
Obviously, Riemannian manifolds 
with parallel Ricci tensor field, especially 
Einstein manifolds, are 
spaces of harmonic curvature.

\subsection{Curvature tensor fields}\label{sec:1.3}
On a statistical manifold $(M,g,\nabla)$, 
the curvature tensor fields
\[
R(X,Y)=[\nabla_{X},\nabla_{Y}]-\nabla_{[X,Y]},
\quad 
R^{*}(X,Y)=[\nabla^{*}_{X},\nabla^{*}_{Y}]-\nabla^{*}_{[X,Y]}
\]
of $\nabla$ and $\nabla^{*}$ are related by
\[
g(R(X,Y)Z,W)+g(Z,R^{*}(X,Y)W)=0.
\]
Namely, the adjoint operator $R(X,Y)^{*}$ of the endomorphism 
field $R(X,Y)$ with respect to $g$ is $-R^{*}(X,Y)$. 
Thus, when $\nabla$ is flat, 
\textit{i.e.}, $R=0$, $\nabla^{*}$ is also flat.

The curvature tensor field $R$ of $\nabla$ is related to 
the Riemannian curvature $R^{g}=R^{(0)}$ of $g$ by 
\[
R(X,Y)Z=
R^g(X,Y)Z+\frac{1}{4}[K(X),K(Y)]Z
-\frac{1}{2}\left(
(\nabla^{g}_{X}K)(Y,Z)-(\nabla^{g}_{Y}K)(X,Z)
\right).
\]

\subsection{Ricci tensor fields}\label{sec:1.4}
Let us denote by $\mathrm{Ric}$ the \emph{Ricci tensor field} of 
$(M,\nabla)$. It should be remarked that 
$\mathrm{Ric}$ is not 
necessarily symmetric. Here we recall the following characterization.

\begin{Proposition}
On a manifold $M$ an torsion free linear 
connection $\nabla$ has symmetric 
Ricci tensor field if and only if 
there exists a volume element $\omega$ parallel with respect 
to $\nabla$.
\end{Proposition}

\begin{Remark}
{\rm Let $M$ be a manifold. A pair $(\omega,\nabla)$ consisting 
of a volume element $\omega$ and a linear connection $\nabla$ on $M$ 
is said to be an \emph{equiaffine strcture} if $\omega$ is parallel with respect to 
$\nabla$. The resulting triplet $(M,\omega,\nabla)$ is called an 
\emph{equiaffine manifold}.
}
\end{Remark}

\begin{Example}[Equiaffine immersions]
{\rm Let $\mathbb{A}^{n+1}=(\mathbb{R}^{n+1},D,\det)$ be 
the \emph{equiaffine} $(n+1)$-\emph{space}, 
that is, the Cartesian $(n+1)$-space $\mathbb{R}^{n+1}$ equipped 
with the natural equiaffine structure $(D,\det)$. Here $D$ is the natural  
flat linear connection and $\det$ is the standard volume 
element.   
Let $\iota:M\to\mathbb{A}^{n+1}$ be an affine immersion of an $n$-manifold $M$ into 
$\mathbb{A}^{n+1}$ with globally defined 
transversal vector field $\xi$. 
Under the assumption, the pair $(\iota,\xi)$ 
induces a torsion free linear connection 
$\nabla$ on $M$ via the \emph{Gauss formula}:
\[
D_{X}Y=\nabla_{X}Y+h(X,Y)\xi, \quad 
X,Y\in\mathfrak{X}(M).
\]
The connection $\nabla$ is called the 
\emph{induced connection}. The tensor field $h$ of type $(0,2)$ on $M$ is symmetric and 
called the \emph{affine fundamental form}.
Let us introduce a volume element $\omega$ on $M$ by
\[
\omega(X_1,X_2,\dots,X_n):=\det(X_1,X_2,\dots,X_n,\xi).
\]
When $(\nabla,\omega)$ is an equiaffine structure, then 
$(\iota,\xi)$ is called an \emph{equiaffine immersion}.
Hereafter we assume that $(\iota,\xi)$ is equiaffine.

The rank of $h$ is independent 
of the choice of $\xi$. 
Thus the notion of non-degeneracy of $h$ is well-defined. 
Now let us assume that $(\iota,\xi)$ is non-degenerate and 
$h$ is positive definite at a point, then 
it is positive definite on whole $M$. Hence 
$h$ is a Riemannian metric on $M$. 
In such a case $(\iota,\xi)$ is said to be 
\emph{locally strongly convex}. 
Moreover the Codazzi equation 
is given by 
\[
(\nabla_{X}h)(Y,Z)=(\nabla_{Y}h)(X,Z),
\quad 
X,Y,Z\in\mathfrak{X}(M).
\]
Thus $(M,h,\nabla)$ is a statistical manifold. 
For more information on 
affine immersion, see \cite{NS}. 

For a prescribed statistical $n$-manifold $(M,h,\nabla)$ 
if there exits an equiaffine immersion $\iota:M
\to\mathbb{A}^{n+1}$ such that $\nabla$ coincides with 
the induced connection and $h$ coincides with the affine fundamental form, 
then $(\iota,\xi)$ is called a \emph{realization} of $(M,h,\nabla)$ in 
$\mathbb{A}^{n+1}$.
}
\end{Example}

\begin{Example}[Locally strongly convex hypersurfaces]
{\rm Let $(\mathcal{M}^{n+1}(c),\tilde{g})$ be a space form of constant curvature $c$ with 
Levi-Civita connection $\nabla^{\tilde{g}}$. 
An orientable hypersurface $\iota:M\to\mathcal{M}^{n+1}(c)$ with unit normal $\nu$ 
inherits a Riemannian metric $g=\iota^{*}\tilde{g}$ and Levi-Civita connection $\nabla^{g}$ from 
$(\tilde{g},\nabla^{\tilde{g}})$ through the Gauss formula:
\[
\nabla^{\tilde{g}}_{X}Y=\nabla^{g}_{X}Y+h(X,Y)\nu, \quad 
X,Y\in\mathfrak{X}(M).
\]
Assume that the hypersurface has 
positive Gauss-Kronecker curvature, 
\textit{i.e.}, $h$ is positive definite, then 
$(M,h,\nabla^{g})$ is a statistical manifold.
}
\end{Example}

\subsection{Apolarity}\label{sec:1.5}
Assume that a statistical manifold $M=(M,g,\nabla)$ is orientable and 
denote by $dv_g$ the volume element 
determined by the metric $g$ and the fixed orientation. Then we have 
(see \cite[p.~361]{Op2015}):
\[
\nabla_{X}(dv_{g})=-\tau(X)\,dv_g,
\]
where the $1$-form $\tau$ is given by
\[
\tau(X)=-\frac{1}{2}\,\mathrm{tr}_{g}(\nabla_{X}g)=-\frac{1}{2}\mathrm{tr}_{g}(K_X)
\]
and $\mathrm{tr}_g$ is the metrical trace with respect to $g$. 
Note that the vector field metrically dual to $\tau$  with respect to $g$ is 
given by
\[
E=-\frac{1}{2}\,\mathrm{tr}_{g}K.
\]

\begin{Proposition}
On an oriented statistical manifold $(M,g,\nabla)$, 
the structure $(dv_g,\nabla)$ is equiaffine if and only if 
$\tau=0$.
\end{Proposition}

\subsection{The $\alpha$-connections}\label{sec:1.6}
Motivated by Amari-Chentsov $\alpha$-connections 
on statistical models, 
the notion of $\alpha$-connection 
on an arbitrary statistical manifold $(M,g,\nabla)$ 
is defined in the following mannar:
\begin{equation}
\nabla^{(\alpha)}_XY =\nabla^{g}_XY-\frac
{\alpha}{2}K(X)Y
\end{equation}
The linear connection $\nabla^{(\alpha)}$ is called the
$\alpha$-\emph{connection} \cite{AN}. 

Note that $\nabla^{(1)}=\nabla$
and $\nabla^{(-1)}=\nabla^*$. 
The covariant derivative of $g$ relative to $\nabla^{(\alpha)}$  is
\[
(\nabla^{(\alpha)}_X g)(Y,Z)=\alpha\, C(X,Y,Z).
\]
Thus $(M, g,\nabla^{(\alpha)})$ is statistical for all
$\alpha \in \mathbb{R}$.

The curvature tensor field 
$R^{\alpha}$ of $\nabla^{(\alpha)}$ is 
given by
\[
R^{\alpha}(X,Y)Z=
R^g(X,Y)Z+\frac{\alpha^2}{4}[K(X),K(Y)]Z
-\frac{\alpha}{2}\left(
(\nabla^{g}_{X}K)(Y,Z)-(\nabla^{g}_{Y}K)(X,Z)
\right).
\]

\subsection{The scalar curvature}\label{sec:1.7}
Denote by $\mathrm{Ric}^{*}$ the 
Ricci tensor field of $\nabla^{*}$, then 
we have 
$\mathrm{tr}_g\mathrm{Ric}=\mathrm{tr}_{g}\mathrm{Ric}^{*}$. 
Based on this fact, the \emph{scalar curvature} $\rho$ of $(M,g,\nabla)$ 
is defined as $\rho=\mathrm{tr}_g\mathrm{Ric}$.

On an oriented statistical manifold $(M,g,\nabla)$ equipped with 
the volume element $dv_g$, we know the following result \cite[Proposition 3.2]{Op2015}.

\begin{Proposition}
On an oriented statistical manifold $(M,g,\nabla)$, 
the Ricci tensor field $\mathrm{Ric}$ of $\nabla$ is symmetric 
if and only if $d\tau=0$. 
\end{Proposition}

\subsection{The conjugate symmetry}\label{sec:1.8}
Take a tangent plane $\Pi=X\wedge Y$ 
at a point $p$ of a statistical manifold, 
the quantity
\[
\frac{g(R(X,Y)Y,X)}{g(X,X)g(Y,Y)-g(X,Y)^2}
\]
is not well-defined object, since $\nabla g\not=0$, in general. 
Here we recall the following notion \cite{Lau}: 

\begin{Definition}{\rm 
A statistical manifold $(M,g,\nabla)$ is said to be 
\emph{conjugate symmetric} provided that $R=R^{*}$ \cite{Lau}.
}
\end{Definition}
One can see that if $M$ is conjugate symmetric then 
$\mathrm{Ric}=\mathrm{Ric}^{*}$ and 
$\mathrm{Ric}$ is symmetric 
\cite[Proposition 3.2]{Op2015}.
Moreover, on a conjugate symmetric 
statistical manifold $M$, the 
\emph{sectional curvature} $\mathsf{K}$ of a tangent plane 
$\Pi\subset T_{p}M$ can be defined by 

\[
\mathsf{K}(\Pi)=\mathsf{K}(X\wedge Y)=\frac{g(R(X,Y)Y,X)}{g(X,X)g(Y,Y)-g(X,Y)^2}.
\]
The difference $R^{*}-R$ is 
computed as
\[
R^{*}(X,Y)Z-R(X,Y)Z=(\nabla^{g}_{X}K)(Y,Z)-(\nabla^{g}_{Y}K)(X,Z).
\]
Thus $M$ is conjugate symmetric if and only if 
$\nabla^{g}K$ is totally symmetric. 
The conjugate symmetry is characterized as follows:

\begin{Lemma}[\cite{BNS,Op2015}]\label{lem:symmetric}
The following statements are mutually equivalent for every 
statistical manifold $(M,g,\nabla)$.
\begin{enumerate}
\item $M$ is conjugate symmetric.
\item $\nabla^{g} C$ is totally symmetric.
\item $\nabla C$ is totally symmetric.
\item $\nabla^{g} K$ is totally symmetric.
\item For any fixed $X$ and $Y$, $(Z,W)\longmapsto g(R(X,Y)Z,W)$ is a $2$-form. 
\end{enumerate}
\end{Lemma}

\begin{Remark}[Statistical sectional curvature]{\rm 
On a statistical manifold $M$, the 
\emph{statistical curvature tensor field}
$R^{\mathsf{S}}$ is introduced 
by Furuhata and Hasegawa \cite{FH}:
\[
R^{\mathsf S}(X,Y)Z:=
\frac{1}{2}(R(X,Y)Z+R^{*}(X,Y)Z),
\quad X,Y,Z\in\mathfrak{X}(M).
\]
See a comment by \cite[p.~91]{Mirja} on 
Kurose's work.

The statistical curvature tensor 
field 
has the form:
\[
R^{\mathsf S}(X,Y)Z=
R^{g}(X,Y)Z
+\frac{1}{4}[K(X),K(Y)]Z.
\]
Here $R^{g}$ is the Riemannian curvature 
of $g$ and $K$ is the skewness operator.

The \emph{statistical sectional curvature} 
$\mathsf{K}^{\mathsf S}(\Pi)$ of a 
tangent plane $\Pi\subset T_{p}M$ is defined by (see Opozda \cite{Op})
\[
\mathsf{K}^{\mathsf S}(\Pi)=\frac{g(R^{\mathsf S}(X,Y)Y,X)}{g(X,X)g(Y,Y)-g(X,Y)^2},
\]
where $\{X,Y\}$ is a basis of $\Pi$. When 
$M$ is conjugate-symmetric, 
$R=R^{*}=R^{\mathsf S}$ and hence 
the statistical sectional curvature coincides with 
sectional curvature. 

A statistical manifold 
$M$ is said to be a 
space of constant 
statistical 
sectional curvature $c$ if 
\[
R^{\mathsf S}(X,Y)Z=c\{g(Y,Z)X-g(Z,X)Y\}.
\]
Note that 
Douhira \cite{Doh} introduced $\mathsf{K}^{\mathsf S}$ independently and 
called the 
\emph{Codazzi bisectional curvature}. It should be 
pointed out, in K{\"a}hler geometry, 
the notion of (holomorphic) bisectional curvature was introduced 
in the following mannar (see \cite{GoKob,SiuYau}):

Let $(M,g,J)$ be a K{\"a}hler manifold.
Then at a point $p\in M$, take two holomorphic planes 
$\Pi_1=X\wedge JHX$ and $\Pi_2=Y\wedge JY$. Here we choose 
$X$ and $Y$ as unit vectors. Then the 
bisectional curvature $H(\Pi_1,\Pi_2)$ is defined by
\[
H(\Pi_1,\Pi_2)=R(X,JX,JY,Y).
\]
Obviously, when $\Pi_1=\Pi_2$, 
$H(\Pi_1,\Pi_1)$ is the holomorphic 
sectional curvature $K^{g}(X\wedge JX)$.
}
\end{Remark}
\subsection{Statistical manifolds of constant curvature}\label{sec:1.9}
A statistical manifold $M$ is said to be of 
\emph{constant curvature} $k$ if it satisfies
\[
R(X,Y)Z=k(X\wedge_{g}Y)Z=k\{g(Y,Z)X-g(Z,X)Y\}
\]
for some constant $k$.


\begin{Theorem}[\cite{KO}]
For a statistical manifold 
$(M,g,\nabla)$, the following three properties are mutually equivalent
{\rm:}
\begin{enumerate}
\item It has a constant curvature.
\item It is a projectively flat and conjugate symmetric.
\item It is a projectively flat and 
has self-conjugate Ricci operator.
\end{enumerate}
\end{Theorem}

We refer to the readers 
Amari and Nagaoka's textbook \cite{AN}
for general theory of statistical manifolds.   

\section{Homogeneous statistical manifolds}\label{sec:2}
\subsection{Ambrose-Singer connections}\label{sec:2.1}
A Riemannian manifold $(M,g)$ is said to be a 
\emph{homogeneous Riemannian space} if there exits a Lie group $G$ of 
isometries which acts transitively on $M$.

More generally, $M$ is said to be \emph{locally homogeneous Riemannian 
space} if for each $p$, $q\in M$, 
there exists a local isometry which sends $p$ to 
$q$.

Without loss of generality, we may assume that 
a homogeneous Riemannian space $(M,g)$ is reductive.

Ambrose and Singer \cite{AS} 
gave an \emph{infinitesimal characterization} of 
local homogeneity of Riemannian manifolds. 
To explain their characterization we recall the following notion:

\begin{Definition}{\rm
A \emph{homogeneous Riemannian structure} $S$ on $(M,g)$ is
a tensor field of type $(1,2)$ which satisfies
\begin{equation}
\tilde{\nabla}{g}=0,
\quad 
\tilde{\nabla}{R}=0,
\quad 
\tilde{\nabla}{S}=0.
\end{equation}
Here $\tilde{\nabla}$ is a linear connection on $M$ defined
by $\tilde{\nabla}=\nabla^{g}+S$. The linear connection $\tilde{\nabla}$ is called 
the \emph{Ambrose-Singer connection}.
}
\end{Definition}

Let $(M,g)=G/H$ be a homogeneous Riemannian space.
Here $G$ is a connected Lie group acting transitively
on $M$ as a group of isometries. 
Without loss of generality we can assume that $G$ acts 
\emph{effectively} on $M$. 

The closed subgroup $H$ is the isotropy subgroup
of $G$ at a point $o\in M$ which will be called the \emph{origin} of $M$. 
For any $a\in G$, 
the translation $\tau_{a}$ by $a$ is 
a diffeomorphism on $M$ defined by $(ab)H$. 
Denote by $\mathfrak{g}$ and $\mathfrak{h}$ the Lie algebras of $G$ and
$H$, respectively. Then there exists a linear subspace 
$\mathfrak{m}$ of $\mathfrak{g}$ which is $\mathrm{Ad}(H)$-invariant.
If $H$ is connected, then $\mathrm{Ad}(H)$-invariant property 
of $\mathfrak{m}$ is equivalent to
the condition $[\mathfrak{h},\mathfrak{m}]\subset
\mathfrak{m}$, \textit{i.e.}, $G/H$ is reductive.
The tangent space 
$\mathrm{T}_{p}M$ of $M$ at a point 
$p=a\cdot o$ is identified with $\mathfrak{m}$ via the isomorphism
\[
\mathfrak{m}\ni
X
\longleftrightarrow 
X^{*}_{p}=\frac{\mathrm{d}}{\mathrm{d}t}\biggr \vert_{t=0}
\tau_{\exp(\mathrm{Ad}(a)X)}(p)
\in\mathrm{T}_{p}M.
\]
Then the canonical connection 
$\tilde{\nabla}=\nabla^{\mathrm c}$ is given by
\[
(\tilde{\nabla}_{X^*}Y^{*})_o
=-([X,Y]_{\mathfrak m})^{*}_o,
\quad 
X,Y 
\in \mathfrak{m}.
\]
For any vector $X\in\mathfrak{g}$, 
the $\mathfrak{m}$-component of $X$ is 
denoted by $X_{\mathfrak{m}}$.  
One can see the difference tensor field $S=\tilde{\nabla}-\nabla$
is a homogeneous Riemannian structure.
Thus every homogeneous Riemannian space admits homogeneous Riemannian structures.

Conversely, let $(M,S)$ be a simply connected Riemannian manifold with a
homogeneous Riemannian structure. Fix a point $o\in M$ and put 
$\mathfrak{m}=T_{o}M$. Denote by $\tilde{R}$ the curvature of the 
Ambrose-Singer connection $\tilde{\nabla}$. Then the holonomy algebra
$\mathfrak{h}$ of $\tilde{\nabla}$ 
the Lie subalgebra of the Lie algebra 
$\mathfrak{so}(\mathfrak{m},g_o)$ 
generated by the curvature 
operators $\tilde{R}(X,Y)$ with 
$X$, $Y\in\mathfrak{m}$.

Now we define a Lie algebra structure on the direct sum 
$\mathfrak{g}=\mathfrak{h}\oplus \mathfrak{m}$
by
\begin{align*}
[U,V]=& UV-VU,\\
[U,X]=& U(X),\\
[X,Y]=& -\tilde{R}(X,Y)-S(X)Y+S(Y)X
\end{align*}
for all $X$, $Y\in \mathfrak{m}$ and $U$, $V\in \mathfrak{h}$.

Now let $\tilde{G}$ be the simply connected Lie group with
Lie algebra $\mathfrak{g}$. Then $M$ is a coset manifold $\tilde{G}/\tilde{H}$, where
$\tilde{H}$ is a Lie subgroup of $G$ with Lie algebra $\mathfrak{h}$.
Let $\Gamma$ be the set of all elements in $G$ which act trivially on $M$. Then
$\Gamma$ is a discrete normal subgroup of $\tilde{G}$ and $G=\tilde{G}/\Gamma$
acts transitively and effectively on $M$ as an isometry group.
The isotropy subgroup $H$ of $G$ at $o$ is $H=\tilde{H}/\Gamma$.  
Hence $(M,g)$ is a homogeneous Riemannian space with coset
space representation $M=G/H$.

\begin{Theorem}[\cite{AS}]
A Riemannian manifold $(M,g)$ with a homogeneous Riemannian structure $S$
is locally homogeneous. 
\end{Theorem}

Here we recall the following result due to Sekigawa and Kirichenko. 
\begin{Theorem}[\cite{Kir,Sekigawa}]
Let $(M,g)$ be a Riemannian manifold equipped with a 
$G$-structure determined by a tensor field $F$ compatible 
to the metric $g$. 
Then the structure $(g,F)$ is locally homogeneous if and only if 
there exits a homogeneous Riemannian structure $S$ such that $F$ is parallel 
with respect to the Ambrose-Singer connection $\tilde{\nabla}=\nabla^{g}+S$. 
\end{Theorem}

\subsection{Homogeneous statistical structure}\label{sec:2.2}
Here we introduce the notion of homogeneous statistical manifold in the 
following manner.
\begin{Definition}
{\rm Let $(M,g,\nabla)$ be a
statistical manifold and $G$ a Lie group. 
Then $(M,g,\nabla)$ is said 
to be a 
\emph{homogenous statistical manifold} (or 
\emph{homogeneous statistical space}) if 
$G$ acts transitively on $M$ and the action is 
isometric with respect to $g$ and affine with 
respect to $\nabla$. 
}
\end{Definition}

\begin{Example}
{\rm Let $G$ be a Lie group equipped with a 
Riemannian metric $g$ and an affine connection $\nabla$.
If both  $g$ and $\nabla$ are invariant under left 
translations, then $(G,g,\nabla)$ is a 
homogenous statistical manifold. The resulting 
homogenous statistical manifold is called a 
\emph{statistical Lie group}.
}
\end{Example}

From Kirichenko-Sekigawa theorem we obtain the following fact.
\begin{Proposition}\label{prop:2.2.1}
Let $(M,g,C)$ be a homogeneous statistical manifold, then 
there exits a homogeneous Riemannian structure $S$ satisfying $\tilde{\nabla}C=0$. 

Conversely, let $(M,g,C)$ be a statistical manifold admitting a homogeneous Riemannian structure $S$ 
satisfying $\tilde{\nabla}C=0$, then $(M,g,C)$ is locally homogeneous.
\end{Proposition}
On a statistical manifold $M=(M,g,C)$, a 
\emph{homogeneous statistical structure} is a 
homogeneous Riemannian structure $S$ satisfying 
$\tilde{\nabla}C=0$.


\subsection{Sasakian statistical manifolds}\label{sec:2.3}
Here we exhibit examples of odd-dimensional homogeneous statistical manifolds.
For this purpose we recall the following notion:
\begin{Definition}{\rm 
A one-form $\eta$ on a $(2n+1)$-dimensional manifold $M$ is said to be a 
\emph{contact form} if it satisfies $(d\eta)^n\wedge \eta\not=0$. 
A manifold $M$ equipped with a contact form is called a 
\emph{contact manifold}.
}
\end{Definition}
The notion of contact form has its origin in 
analytical mechanics and optics (Huygens' principle, 
wave fronts).
We can see that Lie \cite{Lie} studied contact transformation 
\newline
\noindent(Ber{\"u}hrungstransformation). Moreover 
contact manifolds play fundamental role 
in the study of Thermodynamics \cite{HL}.

On a contact manifold $(M,\eta)$, there exits a 
unique vector field $\xi$ (called the \emph{Reeb vector field}) 
satisfying $\eta(\xi)=1$ and $d\eta(\xi,\cdot)=0$. Moreover there exits an
endomorphism field $\varphi$ and a Riemannian metric $g$ 
satisfying 
\[
\varphi^{2}=-\mathrm{Id}+\eta\otimes\xi,\quad 
d\eta=g(\cdot,\varphi),\quad 
g(\varphi \cdot,\varphi \cdot)=g-\eta\otimes\eta.
\]
The quartet $(\varphi,\xi,\eta,g)$ of structure tensor fields  is 
often called a \emph{contact Riemannian structure} on $M$.

Let us introduce a Riemannian cone 
\[
C(M)=(\mathbb{R}^{+}\times M, dr^2+r^2g,d(r^2\eta))
\]
equipped with a $2$-form $d(r^2\eta)$. 
Then $M$ is said to be a 
\emph{Sasakian manifold} if its Riemannian cone 
is a K{\"a}hler manifold.

Sasakian manifolds have been paid much attention 
of mathematicians as well as 
theoretical physicists. 
Indeed, the so-called AdS/CFT correspondences 
is interpreted as a conjecture 
that 
the type IIB string theory on 
the product manifold $\mathrm{AdS}_5\times \mathrm{SE}^5$ 
of anti de Sitter 5-space time 
$\mathrm{AdS}_5$ (which is Lorentz-Sasakian) and 
Sasaki-Einstein $5$-manifold 
$\mathrm{SE}^5$ is 
equivalent to 4D N=1 superconformal quiver 
gauge theory.

On a Sasakian manifold $M$, the Reeb vector field 
is a Killing vector field. Thus geodesics 
which are initially orthogonal to $\xi$ remain ortogonal to $\xi$. 
Such a geodesic is called a $\varphi$-geodesic.
One can consider local reflections around $\varphi$-geodesics. 
Those reflection are called $\varphi$-geodesic symmetries.
More precisely, a local diffeomorphism $s_p$ of a Sasakian manifold 
$M$ is said to be a $\varphi$-\emph{geodesic symmetry} with 
base 
point $p\in M$ if for each 
$\varphi$-geodesic $\gamma(s)$ such that the initial 
point $\gamma(0)$ lies in the trajectory of $\xi$ passing 
through $p$, $s_p$ sends $\gamma(s)$ to $\gamma(-s)$ for any $s$.
Since the points of the Reeb flow through $p$ are fixed 
by $s_p$, the $\varphi$-geodesic symmetry 
$s_p$ is represented as 
\[
s_{p}=\exp_{p}\circ (\mathrm{Id}_{p}+2\eta_{p}\otimes\xi_p)
\circ \exp_{p}^{-1}
\]
on the normal neighborhood of $p$.

The notion of locally $\varphi$-symmetric space
was introduced by Takahashi \cite{Takahashi} and 
reformulated by Buken and Vanhecke \cite{BuV}
in the folloing mannar:
\begin{Definition}
A locally $\varphi$-\emph{symmetric space} is a
Sasakian manifold whose local $\varphi$-geodesic 
symmetries are isometric.
\end{Definition}
A locally $\varphi$-symmetric space is 
said to be a \emph{Sasakian} $\varphi$-\emph{symmetric space} 
if $\xi$ is a complete vector field and 
all the $\varphi$-geodesic symmetries are global 
isometries. On a Sasakian $\varphi$-symmetric space,
$M$ the Lie group $G$ of automorphims acts 
transitively. Thus $M$ is a homogeneous space of $G$.

One can see that complete and simply connected 
locally $\varphi$-symmetric spaces are Sasakian 
$\varphi$-symmetric space. Moreover 
it is known that every Sasakian $\varphi$-symmetric 
space $M$ is a principal circle bundle 
or principal line bundle 
over a Hermitian symmetric space.

From homogeneous geometric viewpoint, we 
emphasize that  
Sasakian $\varphi$-symmetric spaces are 
naturally reductive.   

Let us introduce a one-parameter family of linear connections 
on a (general) Sasakian manifold $M$:
\[
\tilde{\nabla}^{r}_{X}Y=\nabla^{g}_{X}Y+A^{r}(X)Y,
\quad A^{r}(X)Y=g(X,\varphi Y)\xi-r\eta(X)\varphi Y
+\eta(Y)\varphi X.
\]
The connection $\tilde{\nabla}^r$ with $r=1$ is called the 
\emph{Okumura connection}. On the other hand, $\nabla^r$ 
with $r=-1$ is called the 
\emph{Tanaka-Webster connection} (studied in 
CR-geometry). 
Every Sasakian manifold satisfies
\[
\tilde{\nabla}^{r}\varphi=0,
\quad 
\tilde{\nabla}^{r}\xi=0,
\quad
\tilde{\nabla}^{r}\eta=0,
\quad
\tilde{\nabla}^{r}g=0.
\]
Let us denote by $\tilde{R}^r$ the curvature tensor field 
of $\tilde{\nabla}^{r}$. Then
one can check that $M$ is locally 
$\varphi$-symmetric if and only if $\tilde{\nabla}^{r}R^{r}=0$
for some $r$ (and in turn all $r$).  
On a locally $\varphi$-symmetric space $M$, 
$A^r$ is a homogeneous Riemannian structure. 
In particular $A^{1}$ is a naturally reductive 
homogeneous structure.

\begin{Proposition}[\cite{BV}]
Let $M$ be a Sasakian $\varphi$-symmetric space. Then 
$M$ is a naturally reductive homogeneous space 
whose canonical connection is $\tilde{\nabla}^1=\nabla^g+A^1$.
\end{Proposition}

Now let us turn our attention to statistical manifolds again.
The notion of statistical Sasakian manifold is introduced as
follows \cite{FHOSS}:

\begin{Definition}
A \emph{Sasakian statistical structure} 
on an odd-dimensional manifold $M$ is a 
quintet $(\varphi,\xi,\eta,g,\nabla)$ such that
\begin{itemize}
\item $(M,\varphi,\xi,\eta,g)$ is a Sasakian manifold.
\item $(\nabla,g)$ is a statistical structure.
\item The skewness operator $K$ satisfies $K(X)\varphi Y+\varphi K(X)Y=0$.
\end{itemize}
\end{Definition}
Let 
$(M,\varphi,\xi,\eta,g,\nabla)$ be a 
Sasakian statistical manifold and  $\nabla^{(\alpha)}$ the 
$\alpha$-connection $\nabla^{(\alpha)}$. 
Then one can check that $(M,\varphi,\xi,\eta,g,\nabla^{(\alpha)})$ is 
Sasakian statistical for any $\alpha\in\mathbb{R}$.

Now let us take a Sasakian $\varphi$-symmetric space $(M,\varphi,\xi,\eta,g)$. 
We consider a trilinear form 
$C=\eta\otimes\eta\otimes\eta$. Then the corresponding skewness operator 
is $K(X)Y=\eta(X)\eta(Y)\xi$. Thus we obtain a statistical 
structure $(g,\nabla)$ with $\nabla=\nabla^g-\frac{1}{2}K$. 
Note that $(\varphi,\xi,\eta,g,\nabla)$ is Sasakian statistical.
Moreover, since $\tilde{\nabla}^r\eta=0$, 
we have $\tilde{\nabla}^r C=0$. Thus 
the statistical structure $(g,\nabla)$ is homogeneous.

\begin{Proposition}\label{prop:2.5}
Let $M$ be a Sasakian $\varphi$-symmetric space. 
Introduce a one-parameter family of linear connections on $M$ by
\[
\nabla^{(\alpha)}_{X}Y=\nabla^{g}_{X}Y-\frac{\alpha}{2}\eta(X)\eta(Y)\xi,
\quad \alpha\in\mathbb{R}.
\]
Then $(M,\varphi,\xi,\eta,g,\nabla^{(\alpha)})$ is a homogeneous 
Sasaklian statistical manifold.
\end{Proposition}
Here we exhibit only two (well known) compact examples. 

\begin{Example}[Odd dimensional sphere]
{\rm Let $\mathbb{S}^{2n+1}\subset\mathbb{C}^{n+1}$ be 
the unit $(2n+1)$-sphere. Then the K{\"a}hler structure 
of the complex Euclidean $(n+1)$-space induces a 
Sasakian structure on $\mathbb{S}^{2n+1}$. Then 
the resulting Sasakian manifold $\mathbb{S}^{2n+1}$ is a
a Sasakian $\varphi$-symmetric space 
represented as $\mathbb{S}^{2n+1}=
\mathrm{SU}_{n+1}/\mathrm{SU}_n$. 
The Reeb flow constitute a circle group $\mathrm{U}_1$. 
The factor space of $\mathbb{S}^{2n+1}$ under Reeb flows 
is the complex projective space 
$\mathbb{C}P_n=\mathrm{SU}_{n+1}
/\mathrm{S}(\mathrm{U}_1\times\mathrm{U}_n)$ which is 
Hermitian symmetric. The Sasakian sphere 
$\mathbb{S}^{2n+1}=
\mathrm{SU}_{n+1}/\mathrm{SU}_n$ is naturally reductive and 
Sasakian statistical with respect to the cubic form $\alpha \eta\otimes\eta
\otimes\eta$.
}
\end{Example}

\begin{Example}[The Stiefel manifold]
{\rm Let $\mathrm{Sti}_{2}(\mathbb{R}^{n+1})$ be 
the Stiefel manifold of oriented orthonormal $2$-frames 
in Euclidean $(n+1)$-space. The Stiefel manifold 
is a principal circle bundle over the 
Grassmannian manifold $\widetilde{\mathrm{Gr}}_{2}(\mathbb{R}^n)$ 
of oriented $2$-planes in Euclidean $(n+1)$-space.
As is well known $\widetilde{\mathrm{Gr}}_{2}(\mathbb{R}^n)$ is a
Hermitian symmetric space 
$\mathrm{SO}_{n+1}/\mathrm{S}(\mathrm{O}_2\times\mathrm{O}_{n-1})$.
The Stiefel manifold 
$\mathrm{Sti}_{2}(\mathbb{R}^{n+1})=\mathrm{SO}_{n+1}/\mathrm{SO}_{n-1}$ is 
a Sasakian $\varphi$-symmetric space fibered over $\widetilde{\mathrm{Gr}}_{2}(\mathbb{R}^n)$.
Thus $\mathrm{Sti}_{2}(\mathbb{R}^{n+1})$ is Sasakian statistical with respect to the cubic form $\alpha \eta\otimes\eta
\otimes\eta$. 

Note that if we regard $\mathrm{Sti}_{2}(\mathbb{R}^{n+1})$ as 
the unit tangent sphere bundle $\mathrm{U}\mathbb{S}^n$, then 
$\widetilde{\mathrm{Gr}}_{2}(\mathbb{R}^n)$ is identified with 
the space $\widetilde{\mathrm{Geo}}(\mathbb{S}^3)$ of 
all oriented geodesics in $\mathbb{S}^n$. 
In particular, $\mathrm{U}\mathbb{S}^3$ is a trivial sphere bundle 
$\mathbb{S}^3\times\mathbb{S}^2$. 
The unit tangent sphere bundle $\mathrm{U}\mathbb{S}^3$ 
has been paid attention by differential geometers as well as 
theoretical physicists. Indeed, via the pseudo-homothetic deformation of 
the naturally reductive metric, we obtain a Sasaki-Einstein metric 
on $\mathbb{S}^3\times\mathbb{S}^2$ which is called the \emph{Kobayashi-Tanno metric}. 
On the other hand, Gauntlett, Martelli, Sparks and Waldram \cite{GMSW} 
constructed cohomogeneity one 
Sasaki-Einstein metrics on $\mathbb{S}^3\times\mathbb{S}^2$. 
These give rise to solutions of type IIB super-gravity. 

}
\end{Example}

\section{Left invariant  connections on Lie groups}\label{sec:3}
 Let $G$ be a connected 
 Lie group and denote by $\mathfrak{g}$ 
 the Lie algebra of $G$, that is, 
 the tangent space $T_{e}G$ of $G$ at the unit element $ e \in G$. 
 In this section we consider left invariant linear 
 connections on $G$. 
Let $\theta$ be the left invariant Maurer-Cartan form 
 of $G$. By definition, for any tangent vector $X_a$ of $G$ at $a\in G$, we have 
 $\theta _a (X)=(dL_{a})^{-1}_{a} X_a\in\mathfrak{g}$. 
 Here $L_a$ denotes the left translation 
 by $a$ in $G$. 

\subsection{Cartan-Schouten's connections}\label{sec:3.1}
 Take a bilinear map 
 $\mu:\mathfrak{g}\times \mathfrak{g}\to \mathfrak{g}$.
 Then we can define a left invariant linear connection $\nabla^\mu$ on $G$ by
 its value at the unit element $e \in G$ by
\[
\nabla^{\mu}_{X}Y=\mu(X,Y), \ \ X,Y
\in \mathfrak{g}.
\]
\begin{Proposition}[\cite{Nomizu}]\label{allconnections}
Let $\mathcal{B}_\mathfrak{g}$ be the linear 
 space of all $\mathfrak{g}$-valued 
 bilinear maps on $\mathfrak{g}$ and 
 $\mathcal{A}_G$ the affine space of all 
 left-invariant linear connections on $G$.
 Then the map
\[
\mathcal{B}_{\mathfrak g}\ni \mu\longmapsto \nm \in \mathcal{A}_G 
\] 
 is a bijection between $\mathcal{B}_{\mathfrak g}$ 
 and $\mathcal{A}_G$. The torsion $T^\mu$ of $\nm$ is given by
\[
T^{\mu}(X,Y)=-[X,Y]+\mu(X,Y)-\mu(Y,X)
\]
for all $X$, $Y\in \mathfrak{g}$.
\end{Proposition}


\begin{Definition}\label{def:threeconnections}
 We define a one parameter family 
 $\{ {}^{(t)}\nabla \;| \;t\in \mathbb{R}\}$ of 
 left invariant linear connections  $\nt:=\nabla^\mu$ by
$\mu={}^{(t)}\mu$, where
\begin{equation}\label{eq:familyconnection}
{}^{(t)}\mu(X,Y):=\frac{1}{2}(1+t) [X, Y],\ \  X,Y \in \mathfrak{g}.
\end{equation}
It is straightforward to verify that the torsion of ${}^{(t)}\nabla $ is given by 
${}^{(t)}T(X,Y) = t [X,Y]$.

 There are three particular connections in the family $\{\nt\; |\;t \in \mathbb R\}$:
\begin{enumerate}
\item The \emph{canonical connection}: ${}^{(-1)}\nabla $ is defined by
setting  $t =- 1$.

\item The \emph{anti-canonical connection}  : ${}^{(1)}\nabla $
is defined setting $t =1$.

\item The \emph{neutral connection}: ${}^{(0)}\nabla $ is defined by 
setting  $t =0$.
\end{enumerate} 
\end{Definition}

\begin{Remark}\label{rm:threeconnections}
{\rm The canonical connection and anti-canonical connection were 
 defined in \cite{KN2} and  \cite{Agricola}, respectively.
 Among these three connections, $\neutral$ 
 is the only torsion free connection.
 In \cite{KN2}, these three connections $\can$, $\anti$ 
 and $\neutral$ are called \emph{Cartan-Schouten's} 
 $(-)$-\textit{connection}, $(+)$-\textit{connection} 
 and $(0)$-\textit{connection}, 
 respectively. 
}
\end{Remark}

\begin{Remark} 
{\rm  The set of all \emph{bi-invariant} linear connections on $G$ 
 is parametrized by 
\[
\mathcal{B}^{\mathsf{bi}}_{\mathfrak g}=
\left\{ \mu \in \mathcal{B}_{\mathfrak g}\  \vert \ 
\mu(\ad (g)X,\ad (g)Y)=\ad (g)
\mu(X,Y),\ \ \textrm{for any}\ X,Y\in \mathfrak{g}, 
\ g\in G\right\}.
\]
 Laquer \cite{Laquer} proved that for any compact simple 
 Lie group $G$, the set 
 $\mathcal{B}^{\mathsf{bi}}_{\mathfrak g}$ is $1$-dimensional except 
 for the case $G=\mathrm{SU}_n$ with $n\geq 3$. 
 More precisely, the set $\mathcal{A}^{\mathsf{bi}}_{G}$ of all
 bi-invariant linear connections on $G$, except $\mathrm{SU}_n$ ($n\geq 3$), 
 is given by 
\[
\mathcal{A}^{\mathsf{bi}}_{G}=
\{{}^{(t)}\nabla\ \vert \ t\in \mathbb{R}\}.
\]
In case 
$G=\mathrm{SU}_n$ with $n\geq 3$, $\mathcal{B}^{\mathsf{bi}}_{\mathfrak{su}_n}$ 
is $2$-dimensional. The set $\mathcal{B}_{\mathfrak{su}(n)}^{\mathsf{bi}}$ 
is parametrized as
\[
\mathcal{B}_{\mathfrak{su}(n)}^{\mathsf{bi}}
=\{{}^{(t,s)}\mu\ \vert \ t,s\in \mathbb{R}\}
\]
with
${}^{(t,s)}\mu(X,Y)=
\frac{1}{2}(1+t)[X,Y]
+\sqrt{-1}\:s\left(
(XY+YX)-\frac{2}{n}\mathrm{tr}\>(XY)I
\right)$. 
Her $I$ denotes the identity matrix. 
Ikawa \cite{Ikawa} showed that the bi-invariant 
linear connection determined by ${}^{(t,s)}\mu$ with 
$t=0$ and $s=-n/(2\sqrt{n^2+4})$ is Yang-Mills.
}
\end{Remark}

\section{Left invariant metrics on Lie groups}\label{sec:4}
 Here we study Levi-Civita connection of 
 left invariant metrics on Lie groups.

 We equip an inner product $\langle\cdot,\cdot\rangle$ 
 on the Lie algebra $\mathfrak{g}$ of a  Lie group $G$ 
 and extend it to a left invariant Riemannian metric 
 $g=\langle \cdot,\cdot\rangle$ on $G$.
 Here we define a symmetric bilinear map 
 $U:\mathfrak{g}\times\mathfrak{g}\to \mathfrak{g}$ 
 by (\textit{cf.} \cite[Chapter X.3.]{KN2}):
\begin{equation}\label{NR}
 2\langle U(X,Y),Z\rangle 
 =\langle X,[Z,Y]\rangle+\langle Y,[Z,X]\rangle,
 \ \   X,Y,Z \in \mathfrak{g}.
\end{equation}
By using the ad-representation, $U$ is rewritten as
\[
2\langle U(X,Y),Z\rangle 
 =\langle X,\mathrm{ad}(Z)Y\rangle
 +\langle \mathrm{ad}(Z)X,Y\rangle.
\]
This formula implies that 
$g$ is right invariant if and only if $U=0$.
 The Levi-Civita connection $\nabla^{g}$ of 
 $(G,g)$ is given by 
\begin{equation}\label{Levi-CivitaRelation}
\nabla^{g}_{X}Y=\frac{1}{2}[X,Y]+U(X,Y), \;\;X, Y \in \mathfrak g.
\end{equation}
This formula is a variant of well-known Koszul formula.
 Hence $\nabla^{g}$ is a left invariant connection 
 $\nabla^{\mu}$ with the bilinear map
 \[
 \mu(X, Y)=\frac{1}{2}[X,Y] +U(X, Y)
 \]
 The formula \eqref{Levi-CivitaRelation} 
 implies that the left invariant metric 
 $g$ is bi-invariant if and only if $\nabla=\neutral$.

\begin{Proposition}
The Levi-Civita connection $\nabla^{g}$ is a 
left invariant 
connection determined by the bilinear map $\mu$ such that 
\begin{equation}\label{eq:skew-symm}
(\sk \mu)(X,Y)=\frac{1}{2}[X,Y],
\ \ 
(\sym \mu)(X,Y)=U(X,Y).
\end{equation}
Here $\sk \mu$ and $\sym \mu$ are skew-symmetric part and 
symmetric part of $\mu$, respectively.
\end{Proposition}
\section{Left invariant statistical structures on Lie groups}\label{sec:5}
\subsection{Infinitesimal model}\label{sec:5.1}
Let $G$ be a Lie group with a 
left invariant Riemannia metric $g=\langle\cdot,\cdot\rangle$ 
and a left invariant linear connection 
$\nabla^{\mu}$. Then a pair 
$(\langle\cdot,\cdot\rangle,\nabla^{\mu})$ defines a 
left invariant statistical structure if $\nabla^{\mu}$ is of torsion free 
and the 
covariant derivative $C=\nabla^{\mu}g$ of 
$g=\langle\cdot,\cdot\rangle$ by $\nabla^{\mu}$ is totally symmetric. 

Here we recall that 
$\nabla^{\mu}$ is of torsion free 
if and only if 
\[
\mu(X,Y)-\mu(Y,X)=[X,Y].
\]
Namely, 
the skew symmetric part of $\mu$ is
\[
(\mathrm{skew}\>\mu)(X,Y)=\frac{1}{2}[X,Y].
\] 
Thus $\nabla^{\mu}$ has the form:
\[
\nabla^{\mu}_{X}Y=(\mathrm{sym}\>\mu)(X,Y)+\frac{1}{2}[X,Y].
\]
The covariant derivative $C=\nabla^{\mu}g$ 
is computed as:
\begin{align*}
C(X,Y,Z)=&(\nabla^{\mu}_{X}g)(Y,Z)=
-g(\nabla^{\mu}_XY,Z)-g(Y,\nabla^{\mu}_XZ)
\\
=&-\langle \mu(X,Y),Z\rangle
-\langle Y,\mu(X,Z)\rangle,
\\
C(Y,X,Z)=&(\nabla^{\mu}_{Y}g)(X,Z)=
-g(\nabla^{\mu}_YX,Z)-g(X,\nabla^{\mu}_YZ)
\\
=&-\langle \mu(Y,X),Z\rangle
-\langle X,\mu(Y,Z)\rangle,
\\
C(Y,Z,X)=&(\nabla^{\mu}_{Y}g)(Z,X)=
-g(\nabla^{\mu}_YZ,X)-g(Z,\nabla^{\mu}_YX)
\\
=&-\langle \mu(Y,Z),X\rangle
-\langle Z,\mu(Y,X)\rangle.
\end{align*}
Hence the total symmetry condition of $C$ is
\[
\langle \mu(X,Y),Z\rangle
+\langle Y,\mu(X,Z)\rangle
=
\langle \mu(Y,X),Z\rangle
+\langle X,\mu(Y,Z)\rangle.
\]
This formula is rewritten as 
\[
\langle \mu(X,Y)-\mu(Y,X),Z\rangle
=
\langle X,\mu(Y,Z)\rangle
-\langle Y, \mu(X,Z)\rangle.
\]
Next, we get
\begin{align*}
\langle [X,Y],Z\rangle
=&
\langle X,\mu(Y,Z)\rangle
-\langle Y, \mu(X,Z)\rangle
\\
=&
\langle X,(\mathrm{sym}\>\mu)(Y,Z)
+\frac{1}{2}[Y,Z]\rangle
-\langle Y, (\mathrm{sym}\>\mu)(X,Z)
+\frac{1}{2}[X,Z]
\rangle
\\
=&
\langle X,(\mathrm{sym}\>\mu)(Y,Z)
\rangle
-\langle Y, (\mathrm{sym}\>\mu)(X,Z)
\rangle
+\frac{1}{2}
\langle X,[Y,Z]\rangle
-\frac{1}{2}\langle Y, [X,Z]
\rangle
\\
=&
\langle X,(\mathrm{sym}\>\mu)(Y,Z)
\rangle
-\langle Y, (\mathrm{sym}\>\mu)(X,Z)
\rangle
-\frac{1}{2}
\langle \mathrm{ad}(Z)Y,X\rangle
+\frac{1}{2}\langle \mathrm{ad}(Z)X,Y
\rangle.
\end{align*}
Thus we obtain
\begin{equation}\label{leftinvariant}
\langle U(Y,Z),X\rangle
-\langle U(X,Z),Y\rangle
=\langle (\mathrm{sym}\>\mu)(Y,Z),X
\rangle
-\langle (\mathrm{sym}\>\mu)(X,Z),Y
\rangle.
\end{equation}
When $g$ is bi-invariant, 
the total symmetry condition is 
\[
\langle X,(\mathrm{sym}\>\mu)(Y,Z)
\rangle
=\langle Y, (\mathrm{sym}\>\mu)(X,Z)
\rangle.
\]
Now let us put $\nu:=\mathrm{sym}\>\mu$ then our 
result is stated as:
 
\begin{Proposition}\label{prop:5.1}
On a Lie group $G$ equipped with a 
left invariant statistical structure 
$(g,\nabla)$, 
the linear connection $\nabla$ is determined by 
a bilinear map 
$\nu:\mathfrak{g}\times\mathfrak{g}\to
\mathfrak{g}$ satisfying
\[
\langle U(Y,Z),X\rangle
-\langle U(X,Z),Y\rangle
=\langle \nu(Y,Z),X
\rangle
-\langle \nu(X,Z),Y
\rangle.
\]
The linear connection $\nabla$ is represented as 
\[
\nabla_{X}Y=\nu(X,Y)+\frac{1}{2}[X,Y].
\]
\end{Proposition}
The skewness operator $K$ is computed as
\[
-\frac{1}{2}K(X)Y=\nu(X,Y)-U(X,Y).
\]
The total symmetry of $C$ is rewritten as
\[
\langle K(X)Z,Y\rangle
=\langle K(Y)Z,X\rangle.
\]
Henceforth we obtain the following results:
\begin{Corollary}
On a Lie group $G$ equipped with a 
left invariant statistical structure 
$(g,\nabla)$, 
the connection $\nabla$ is determined by 
a bilinear map 
\[
K:\mathfrak{g}\times\mathfrak{g}\to
\mathfrak{g};\ \ (X,Y)\longmapsto K(X)Y
\]
satisfying
\begin{equation}
K(X)Y=K(Y)X,\ \ 
\langle K(X)Z,Y\rangle
=\langle K(Y)Z,X\rangle, 
 \ \ 
\langle K(X)Y,Z\rangle
=\langle Y,K(X)Z\rangle. 
\end{equation}
The connection is represented as 
\[
\nabla_{X}Y=U(X,Y)-\frac{1}{2}K(X)Y+\frac{1}{2}[X,Y].
\]
\end{Corollary}

\begin{Corollary}\label{cor:5.3}
Every bi-invariant statistical structure on $G$ is 
represented by a bi-invariant metric 
$g=\langle\cdot,\cdot\rangle$ and an 
$\mathrm{Ad}(G)$-invariant bilinear map 
$K:\mathfrak{g}\times\mathfrak{g}\to\mathfrak{g}$ satisfying
\[
K(X)Y=K(Y)X,\ \ 
\langle K(X)Z,Y\rangle
=\langle K(Y)Z,X\rangle.
\]
The bi-invariant linear connection $\nabla$ is expressed as
\[
\nabla_XY=-\frac{1}{2}K(X)Y+\frac{1}{2}[X,Y].
\]
\end{Corollary}

To close this subsection we give a reformulation of 
statistical Lie groups in terms of cubic form.

Let $C$ be a symmetric trilinear map on $\mathfrak{g}$. 
By left translations, we can extend $C$ to a left 
invariant cubic form on $G$.
Take a left invariant Riemannian metric $g
=\langle\cdot,\cdot\rangle$ on $G$. 
Then we have a left invariant statistical structure 
$(g,\nabla)$ by
\[
\nabla_{X}Y:=\nabla^{g}_{X}Y-\frac{1}{2}K(X)Y=U(X,Y)-\frac{1}{2}
K(X)Y+\frac{1}{2}[X,Y],
\quad 
X,Y\in\mathfrak{g}.
\]
Here $K$ is defined by 
\[
\langle K(X)Y,Z\rangle=C(X,Y,Z),
\quad 
X,Y,Z\in\mathfrak{g}.
\]

\begin{Corollary}
On a Lie group $G$ equipped with a 
left invariant Riemannian metric statistical structure 
$(g,\nabla)$, 
the connection $\nabla$ is determined by 
a symmetric trilinear form
\[
C:\mathfrak{g}\times\mathfrak{g}\times\mathfrak{g}\to
\mathbb{R};\quad (X,Y,Z)\longmapsto C(X,Y,Z).
\]
The left invariant connection $\nabla$ is represented as 
\[
\nabla_{X}Y=U(X,Y)-\frac{1}{2}K(X)Y+\frac{1}{2}[X,Y],
\]
where the bilinear map $K:\mathfrak{g}\times\mathfrak{g}
\to\mathfrak{g}$ is determined by 
\[
\langle K(X)Y,Z\rangle=C(X,Y,Z),
\quad
X,Y,Z\in\mathfrak{g}.
\] 
\end{Corollary}

\begin{Corollary}\label{cor:5.5}
Every bi-invariant statistical structure on $G$ is 
represented by a bi-invariant metric 
$g=\langle\cdot,\cdot\rangle$ and an 
$\mathrm{Ad}(G)$-invariant symmetric trilinear form  
$C$ on $\mathfrak{g}$.
The bi-invariant connection is expressed as
\[
\nabla_XY=-\frac{1}{2}K(X)Y+\frac{1}{2}[X,Y].
\]
\end{Corollary}

\subsection{Left invariant Codazzi tensor fields}\label{sec:5.2}
Let $G$ be a connected Lie group equipped with 
a left invariant Riemannian metric. Take a 
left invariant Codazzi tensor field $h$. 
Denote by $h^{\sharp}$ the endomorphism field 
metrically dual to $h$, \textit{i.e.},
\[
h(X,Y)=g( h^{\sharp} X,Y),\quad X,Y\in\mathfrak{X}(M).
\]
Then $h^{\sharp}$ is self-adjoint with respect to $g$. 
The eigenvalues of $h^{\sharp}$ are constant and their 
multiplicities are also constant. 
Thus it suffices to investigate $h^{\sharp}$ at the identity $e\in G$: 
Note that the
eigenspaces of $h^{\sharp}$ are Lie subalgebras of $\mathfrak{g}$.
 D'Atri introduced the 
following notion:
\begin{Definition}
{\rm A left invariant Codazzi tensor 
field $h$ is said to be \emph{essential} if 
$\nabla^{g}h\not=0$ and none of the eigenspace 
is an ideal of $\mathfrak{g}$. 
}
\end{Definition}
Under this definition, D'Atri proved the 
following proposition.
\begin{Proposition}[\cite{DAtri}]
If a connected Lie group $G$ equipped with a 
left invariant Riemannian metric admits an essential 
left invariant Codazzi tensor field, then 
$G$ can not be nilpotent or split-solvable over $\mathbb{R}$.
\end{Proposition}
In \cite{DAtri}, D'Atri  constructed essential left invariant 
Codazzi tensor fields on the special unitary group $\mathrm{SU}_2$ equipped with 
a left invariant Riemannian metric. 
Some results of \cite{DAtri} are generalized to 
invariant Codazzi tensor fields on reductive homogeneous Riemannian spaces 
\cite{MarTe}.

\section{Two dimensional Lie groups}\label{sec:6}
In this section we give some explicit examples of $2$-dimensional 
statistical Lie groups.
\subsection{Two-dimensional Lie algebras}\label{sec:6.1}
It is known that every $2$-dimensional non-abelian 
Lie algebra is isomorphic to 
\[
\mathfrak{g}=\left\{
\left(
\begin{array}{cc}
v & u\\
0 & 0
\end{array}
\right)
\ 
\biggr \vert \ u,v\in
\mathbb{R}
\right\}.
\]
The Lie algebra $\mathfrak{g}$ is generated by 
\[
E_1=\left(
\begin{array}{cc}
0 & 1\\
0 & 0
\end{array}
\right) \ \ \mbox{and}\ \ 
E_2=\left(
\begin{array}{cc}
1 & 0\\
0 & 0
\end{array}
\right)
\] 
with commutation relation 
$[E_1,E_2]=-E_1$.
The simply connected and connected Lie group $G$ 
corresponding to $\mathfrak{g}$ is 
\begin{equation}\label{eq:2DG}
G=\left\{
\left(
\begin{array}{cc}
y & x\\
0 & 1
\end{array}
\right)\>\biggr|\>x,y\in
\mathbb{R},\ \ y>0
\right\}.
\end{equation}
The multiplication of $G$ is given by 
\[
\left(
\begin{array}{cc}
y_1 & x_1\\
0 & 1
\end{array}
\right)
\left(
\begin{array}{cc}
y_2 & x_2\\
0 & 1
\end{array}
\right)
=
\left(
\begin{array}{cc}
y_1y_2& x+y_1x_2\\
0 & 1
\end{array}
\right).
\]
Note that $G$ is the Lie group $\mathrm{GA}^{+}_1$ of orientation 
preserving affine transformations of the real line $\mathbb{R}$. 
We denote by the left invariant vector fields 
determined by $e_1$ and $e_2$ by the same letter. One can see that
 \[
E_1= y\frac{\partial}{\partial x},\quad 
E_2=y\frac{\partial}{\partial y}.
 \]
Take a pair $\{\nu_1,\nu_2\}$ of 
positive constants and  we put 
\[
\{e_1:=\nu_{1}^{-1}E_1,\quad  e_2=\nu_{2}^{-1}E_2\}.
\]
Note that $[e_1,e_2]=-\nu_{2}^{-1}e_1$.

We introduce a left-invariant Riemannian metric 
$g$ so that $\{e_1,e_2\}$ is orthonormal 
with respect to it.
Then the induced left-invariant metric $g=g_{\nu_1,\nu_2}$ is 
\begin{equation}\label{eq:Met}
g=\frac{\nu_1^2dx^2+\nu_2^2dy^2}{y^2}.
\end{equation}
Hereafter we denote by $G(\nu_1,\nu_2)$ the solvable Lie 
group $G$ given by \eqref{eq:2DG} equipped with 
the left invariant metric \eqref{eq:Met}.

The symmetric bilinear map $U$ of $G(\nu_1,\nu_2)$ 
defined by \eqref{NR} is computed as
$2\langle U(e_1,e_1),e_1\rangle =0$ and 
\[
 2\langle U(e_1,e_1),e_2\rangle =
\langle e_1, [e_2, e_1]\rangle +\langle e_1, [e_2, e_1]\rangle =
2\nu_{2}^{-1}.
\]
Thus $U(e_1,e_1)=\nu_{2}^{-1}e_2$.
Similarly 
we have
$U(e_1,e_2)=-\nu_{2}^{-1}e_1/2$ 
and $U(e_2,e_2)=0$.
 The Levi-Civita connection is 
described as
\[
\nabla^{g}_{e_1}e_{1}=\frac{1}{\>\nu_2}e_2,
\quad
\nabla^{g}_{e_1}e_{2}=-\frac{1}{\>\nu_2}e_{1},
\quad
\nabla^{g}_{e_2}e_{1}=0,
\quad
\nabla^{g}_{e_2}e_{2}=0.
\]
\begin{Example}[Hyperbolic plane]{\rm 
When $\nu_1=\nu_2=c>0$, the 
resulting Riemannian manifold 
$G(c,c)$ is the hyperbolic plane of curvature $-1/c^2$. 
}
\end{Example}

 \subsection{One parameter subgroups}\label{sec:6.2}
Here we investigate one parameter subgroups 
of $G(\nu_1,\nu_2)$.

Let us consider the one parameter subgroup 
\[
a(t)=\left(
\begin{array}{cc}
1 & t\\
0 & 1
\end{array}
\right).
\]
Then $a(t)$ acts on 
$G(\nu_1,\nu_2)$ as a translation in $x$-direction:
\[
a(t)\cdot(x,y)=(x+t,y).
\]
On the other hand, 
\[
b(t)=\left(
\begin{array}{cc}
t & 0\\
0 & 1
\end{array}
\right)
\]
acts on $G(\nu_1,\nu_2)$ by 
\[
b(t)\cdot (x,y)=(tx,ty).
\]
We compare these actions with usual 
linear fractional transformations on 
the upper half plane.
Let us introduce the complex coordinate $z=x+iy$ on $G(\nu_1,\nu_2)$ and 
identify $G$ with 
the upper half plane 
\[
\mathbb{H}_{+}=\{z=x+yi\>|x,y\in\mathbb{R},\>\>y>0\}.
\] 
Then the linear fractional transformation $\mathrm{T}_{a(t)}$ by 
$a(t)$ on $\mathbb{H}_{+}$ is given by
\[
\mathrm{T}_{a(t)}(z)=\frac{1\cdot z+t}{0\cdot z+1}=z+t.
\] 
On the other hand we have 
\[
\mathrm{T}_{b(t)}(z)=\frac{t\cdot z+0}{0\cdot z+1}=tz.
\] 

\subsection{The statistical manifold of normal distributions}\label{sec:6.3}
The probability density function of a normal distribution 
$N(\mu,\sigma)$ of mean $\mu$ and variance $\sigma^2$ is 
\[
p(x;\mu,\sigma)=
\frac{1}{\sqrt{2\pi\sigma^2}}\exp
\left(
-\frac{(x-\mu)^2}{2\sigma^2}
\right).
\]
Thus the set $\{N(\mu,\sigma)\}$ of all normal distributions is identified with 
the upper half plane 
$\{(\mu,\sigma)\>|\>\mu,\sigma\in\mathbb{R},\,\sigma>0\}$.
The Fisher metric $g=g^{\mathsf F}$ of the 
manifold of normaln distributions is 
computed as 
\[
g^{\mathsf F}=\frac{d\mu^2+2d\sigma^2}{\sigma^2}.
\]
The statistical manifold 
of the normal distribution 
equipped with $g^{\mathsf F}$ and 
e-connection is  
is homogeneous and identified with 
the $2$-dimensional solvable Lie group $G(1,\sqrt{2})$ 
equipped with a left invariant 
connection $\nabla${\rm:}
\[
\nabla_{e_1}e_1=0,
\quad 
\nabla_{e_1}e_2=-\sqrt{2}e_1,
\quad 
\nabla_{e_2}e_1=-\frac{1}{\sqrt{2}}e_1,
\quad 
\nabla_{e_2}e_2=-\sqrt{2}e_2
\]
under the replacement $x=\mu$ and $y=\sigma$. 
The skewness operator is 
given by 
\[
K(e_1,e_1)=\sqrt{2}e_2,
\quad 
K(e_1,e_2)=\sqrt{2}e_1,
\quad 
K(e_2,e_2)=2\sqrt{2}e_2.
\]
The neutral connection 
of the statistical manifold of the normal distribution 
is given by 
\[
{}^{(0)}\nabla_{e_1}e_1=0, 
\quad 
{}^{(0)}\nabla_{e_1}e_2=-
\frac{1}{2\sqrt{2}}e_1,
\quad 
{}^{(0)}\nabla_{e_2}e_1=\frac{1}{2\sqrt{2}}e_1,
\quad 
{}^{(0)}\nabla_{e_2}e_2=0.
\]
The neutral connection is \emph{not} compatible 
to $g$.
The curvature tensor field $R^{\alpha}$ 
of the homogenous statistical manifold $G$ 
of the normal distribution with respect to 
$\alpha$-connection is computed as
\[
R^{\alpha}(e_1,e_2)e_2=\nabla^{\alpha}_{e_1}\nabla^{\alpha}_{e_2}e_2
-\nabla^{\alpha}_{e_2}\nabla^{\alpha}_{e_1}e_2
-\nabla^{\alpha}_{[e_1,e_2]}e_2
=
-\frac{(1+\alpha)(1-\alpha)}{4}e_1.
\]
In particular $\nabla$ and $\nabla^{*}$ are flat. The skewness operator is given by
\[
K^{\alpha}(e_1,e_1)=\sqrt{2}\alpha e_2, \ \ 
K^{\alpha}(e_1,e_2)=\sqrt{2}\alpha e_1, \ \ 
K^{\alpha}(e_2,e_1)=\sqrt{2}\alpha e_1, \ \ 
K^{\alpha}(e_2,e_2)=2\sqrt{2}\alpha e_2.
\]
From these we get
\[
\frac{1}{4}[K^{\alpha}(e_1),K^{\alpha}(e_2)]e_2
=-\frac{\alpha^2}{2}e_1.
\]
Thus the statistical sectional curvature 
with respect to the $\alpha$-connection is
$-\alpha^2/2$.


The conjugate symmetry was introduced from 
purely differential geometric motivation. 
However it characterizes the Amari-Chentsov $\alpha$-connections 
on the statistical manifold of normal distribution 
derived from Statistics. 

\begin{Theorem}[\cite{FIK}]\label{thm:FIK}
Let $G(1,\sqrt{2})$ be the $2$-dimensional solvable Lie group
of the normal distribution equipped the Fischer metric $g$. 
Take an $G$-invariant connection $\nabla$ compatible to the 
Fischer metric. 
Then $(g,\nabla)$ is 
conjugate symmetric if and only if 
$\nabla$ is one of the Amari-Chentsov $\alpha$-connections.
\end{Theorem}
This theorem is still valid for 
statistical Lie groups of multivariate 
normal distributions, see \cite{KO2}.

\subsection{Student's $t$-distributions}\label{sec:6.4}
The Student's $t$-\emph{distribution} $T(\nu;x,y)$  
is a probability distribution determined by the 
pdf \cite{Fisher,Student} (see also \cite{Boland,P}):
\[
f_{\nu}(t)=\frac{\varGamma(\tfrac{\nu+1}{2})}
{\sqrt{\pi\nu\,}\varGamma(\tfrac{\nu}{2})\>y}
\left(
1+\frac{1}{\nu}
\left(
\frac{t-x}{y}
\right)^2
\right)^{-\frac{\nu+1}{2}}
=\frac{1}{\sqrt{\nu}\,B(1/2,\nu/2)y}
\left(
1+\frac{1}{\nu}
\left(
\frac{t-x}{y}
\right)^2
\right)^{-\frac{\nu+1}{2}},
\]
where 
$x\in\mathbb{R}$ (\emph{location parameter}) and 
$y>0$ (\emph{scale parameter}) 
and $\nu\in\mathbb{Z}_{>0}$ (the \emph{degree of freedom}).  
Note that $\varGamma(x)$ is the Gamma function and 
$B(p,q)$ is the Beta function.

The $t$-distribution converges in law to the normal distribution 
$N(x,y)$ under $\nu\to\infty$.
 
We identify the set 
$M_{\nu}=\{T(\nu;x,y)\>|\>x\in\mathbb{R},y>0\}$ of all $t$-distributions 
with freedom $\nu$ with the upper half plane 
$\{(x,y)\in\mathbb{R}^2\>|\>y>0\}$.

The Fisher metric $g$ of $T(\nu;x,y)$ is computed as \cite{CB,CJ,Mori}: 
\begin{equation}\label{eq:FisMetT}
g=\frac{(\nu+1)dx^2+2\nu\,dy^2}{(\nu+3)y^2}
\end{equation}
The Levi-Civita connection $\nabla^g$ of $g$ is given by
\begin{equation}
\nabla^{g}_{\partial_x}\partial_{x}=\frac{\nu+1}{2\nu\,y}\partial_{y},\ \
\nabla^{g}_{\partial_x}\partial_{y}=
\nabla^{g}_{\partial_y}\partial_{x}=-\frac{1}{y}\partial_{x},\ \
\nabla^{g}_{\partial_y}\partial_{y}=-\frac{1}{y}\partial_{y}.
\end{equation}
Here 
\[
\partial_x=\frac{\partial}{\partial x},
\quad
\partial_y=\frac{\partial}{\partial y}.
\]
On the other hand the $\mathrm{e}$-connection is given by
\begin{equation}
\nabla_{\partial_x}\partial_{x}=\frac{3(\nu+1)}{\nu(\nu+5)y}\partial_{y},\ \
\nabla_{\partial_x}\partial_{y}=
\nabla_{\partial_y}\partial_{x}=-\frac{2(\nu+2)}{(\nu+5)y}\partial_{x},\ \
\nabla_{\partial_y}\partial_{y}=-\frac{3(\nu+1)}{(\nu+5)y}\partial_{y}.
\end{equation}
The skewness operator $K=-2(\nabla-\nabla^g)$ is described as
\[
K(\partial_x,\partial_x)=\frac{2(\nu-1)(\nu+1)}{2\nu(\nu+5)y}\,\partial_{y},
\quad
K(\partial_x,\partial_y)=\frac{2(\nu-1)}{(\nu+5)y}\,\partial_{x},
\quad
K(\partial_y,\partial_y)=\frac{4(\nu-1)}{(\nu+5)y}\,\partial_{y}.
\]

The Amari-Chentsov $\alpha$-connection is given by
\begin{align*}
\nabla^{(\alpha)}_{\partial_x}\partial_{x}
=&
\frac{(\nu+1)  \{(\nu+5)-\alpha(\nu-1)\}}
{2\nu(\nu+5)y}\,\partial_{y},
\\
\nabla^{(\alpha)}_{\partial_x}\partial_{y}=&
\nabla^{(\alpha)}_{\partial_y}\partial_{x}=
\frac{-\{(\nu+5)+\alpha(\nu-1)\}}{(\nu+5)y}\,\partial_{x},
\\
\nabla^{(\alpha)}_{\partial_y}\partial_{y}=&
\frac{-\{(\nu+5)+2\alpha(\nu-1)\}}{(\nu+5)y}\,\partial_{y}.
\end{align*}

Note that under the limit $\nu\to \infty$, 
the Fisher metric converges to the Fisher metric 
\[
\frac{dx^2+2dy^2}{y^2}
\]
of the statistical manifold of normal distributions. 
In addition the alpha-connections converge to those of 
normal distributions.

Mori pointed out that the $\alpha$-connection with 
$\alpha=(\nu+5)/(\nu-1)$ coincides 
with e-connection of the normal distribution. 

One can see that the statistical manifold $M_{\nu}$ is 
identified with the solvable Lie group \eqref{eq:2DG}. 
The Fisher metric of the statistical manifold 
of $t$-distributions is obtained by choosing
\[
\nu_1=\sqrt{\frac{\nu+1}{\nu+3}},\quad
\nu_2= \sqrt{\frac{2\nu}{\nu+3}}.
\]
Thus we identify the statistical manifold $M_{\nu}$ with 
$G(\sqrt{\nu+1}/\sqrt{\nu+3},\sqrt{2\nu}/\sqrt{\nu+3})$.

The skewness operator of the statistical manifold 
of the $t$-distributions is given by
\begin{align*}
K(e_1,e_1)=&\frac{\lambda_2}{\lambda_1^2}\>\frac{(\nu-1)(\nu+1)}{\nu(\nu+5)}\,e_2
=\frac{\sqrt{2(\nu+3)}\,(\nu-1)}{\sqrt{\nu}\,(\nu+5)}\,e_2,
\\
K(e_1,e_2)=&\frac{2(\nu-1)}{\lambda_2(\nu+5)}\,e_1
=\frac{\sqrt{2(\nu+3)}\,(\nu-1)}{\sqrt{\nu}\,(\nu+5)}\,e_1,
\\
K(e_2,e_2)=&\frac{4(\nu-1)}{\lambda_2(\nu+5)}\,e_2
=\frac{2\,\sqrt{2(\nu+3)}\,(\nu-1)}{\sqrt{\nu}\,(\nu+5)}\,e_2.
\end{align*}
These formulas show that 
the Amari-Chentsov
$\alpha$-connections of the statistical manifold 
of the $t$-distributions are left invariant. 
In particular, e-connection is left invariant. Hence 
the statistical manifold of $t$-distribution is a statistical Lie group.
\begin{Remark}
{\rm
Every $t$-distribution $T(\nu;x,y)$ is translated 
to $T(\nu;0,1)$ by left translation $(x,y)^{-1}$.
Note that $T(\nu;0,1)$ has mean $0$ and 
variance $\nu/(\nu-2)$ for $\nu>1$ and 
$\infty$ for $1<\nu\leq 2$. 
}
\end{Remark}
The Amari-Chentsov $\alpha$-connection of $M_{\nu}$ 
is given explicitly by
\begin{align*}
\nabla^{(\alpha)}_{e_1}e_1=& 
\frac{\sqrt{\nu+3}}{\sqrt{2\nu}}
\,
\frac{(\nu+5)-\alpha(\nu-1)}{\nu+5}\,e_2,
\\
\nabla^{(\alpha)}_{e_1}e_2=&
-\frac{\sqrt{\nu+3}}{\sqrt{2\nu}}\,
\frac{(\nu+5)+\alpha(\nu-1)}{\nu+5}\,e_1,\\
\nabla^{(\alpha)}_{e_2}e_1=&-
\frac{\alpha\sqrt{\nu+3}}{\sqrt{2\nu}}\,
\frac{\nu-1}{\nu+5}\,e_1,
\\
\nabla^{(\alpha)}_{e_2}e_2=&
-
\frac{2\alpha\sqrt{\nu+3}}{\sqrt{2\nu}}\,
\frac{\nu-1}{\nu+5}\,e_2.
\end{align*}
The curvature tensor field $R^{(\alpha)}$ of $\nabla^{(\alpha)}$ is 
given explicitly by 
\[
R^{(\alpha)}(X,Y)Z=
\frac{\nu+3}{2\nu}
\left(
\frac{\alpha(\nu-1)}{\nu+5}+1
\right)
\left(
\frac{\alpha(\nu-1)}{\nu+5}-1
\right)
\left(
g(Y,Z)X-g(Z,X)Y
\right).
\]
Thus the statistical manifold $M_{\nu}$ is 
a statistical manifold of constant sectional curvature.
We can confirm that $\nabla^{(\alpha)}$ with 
$\alpha=(\nu+5)/(\nu-1)$ is flat. 
One cen see that $M_{\nu}$ is conjugate symmetric.

\begin{Remark}[$q$-normal distributions]
{\rm Motivated by 
Boltmann-Gibbs probability distribution, 
$q$-exponential function $\exp_q$ is introduced 
in Tsallis statistics \cite{Ts} as (see also \cite{ST}): 
\[
\exp_{q}(x)=
\left(1+(1-q)x
\right)^{\frac{1}{1-q}}
\]
for $q\not=1$ and $x$ satisfying 
$1+(1-q)x>0$. Do not confuse with the so-called 
$q$-\emph{analogue} $\exp_q$ of $\exp$:
\[
\exp_{q}(x)=\sum_{n=0}^{\infty}\frac{x^n}{[n]_q!}
=\sum_{n=0}^{\infty}
\frac{(1-q)^nx^n}
{(1-q^n)(1-q^{n-1})\cdots(1-q)}.
\]
Replacing the usual exponential function 
by Tsallis's $q$-exponential 
function 
in the definition 
of statistical model 
of exponential families (and choosing 
$C=0$), 
one obtain the 
statistical model of  
$q$-\emph{exponential family}
\[
M=\left\{
\left.
p(x,\Vec{\theta})\>
=
\exp_{q}\left[
\sum_{i=1}^{n}
\theta^{i}F_{i}(x)-\psi(\theta)
\right]
\>\right|
\>\Vec{\theta}\in\Theta
\right\}.
\]
The role of $q$-exponential 
family in 
statistical 
physics, we refer to 
Naudts's article \cite{Naudts}.

A $q$-\emph{normal distribution}
is an example of $q$-exponential family. 
It is a probability distribution 
determined by the pdf:
\[
p(x,\mu,\sigma)=\max
\left[
\frac{1}{Z_q}
\left\{
1-\frac{1-q}{3-q}\>
\frac{(x-\mu)^2}{\sigma^2}
\right\}^{\frac{1}{1-q}}
,
0
\right],
\]
where
\[
Z_{q}=\left\{
\begin{array}{l}
\frac{\sqrt{3-q}}{\sqrt{1-q}}B\left(
\frac{2-q}{1-q},\frac{1}{2}
\right)\sigma,\quad -\infty <q<1,
\\
\frac{\sqrt{3-q}}{\sqrt{1-q}}
B\left(
\frac{3-q}{2(q-1)},\frac{1}{2}
\right)\sigma,\quad 1<q<\infty
\end{array}
\right.
\] 
The $q$-normal distribution converges in law to the 
normal distribution under the limit $q\to 1$. 
When $1<q<3$, 
under the parameter change 
\[
\nu:=\frac{3-q}{q-1},
\]
The $q$-normal distribution is transformed 
to $t$-distribution. Thus $t$-distribution 
constitues a subclass of $q$-normal 
distributions. For information geometry of 
$q$-normal distributions, see \cite{Matsuzoe,MO,TTM}.
}
\end{Remark}
The statistical manifold of multivariate normal distribution 
is also a statistical Lie group. See \cite{Ko,KO2}.

\section{Three dimensional unimodular Lie groups}\label{sec:7}
In view of Theorem \ref{thm:FIK}, we expect that 
conjugate symmetric homogeneous statistical manifolds 
constitute a particularly nice class 
of statistical manifolds. 
In this section and next section we 
look for $3$-dimensional examples 
of conjugate symmetric statistical Lie groups.

\subsection{The unimodular kernel}\label{sec:7.1}

Let $G$ be a Lie group with Lie algebra $\mathfrak{g}$.
Denote by $\mathrm{ad}$ the 
\textit{adjoint representation} of $\mathfrak{g}$, 
\[
\mathrm{ad}:\mathfrak{g}\to
\mathrm{End}(\mathfrak{g});
\ \ 
\mathrm{ad}(X)Y=[X,Y].
\]
Then one can see that 
$\mathrm{tr}\>\mathrm{ad}$;
\[
X\longmapsto \mathrm{tr}\>\mathrm{ad}(X)
\]
is a Lie algebra homomorphism into the commutative 
Lie algebra $\mathbb{R}$.
The kernel 
\[
\mathfrak{u}=\{X \in \mathfrak{g}
\ \vert \
\mathrm{tr}\> \mathrm{ad}(X)=0\}
\]
of $\mathrm{tr}\>\mathrm{ad}$ is an ideal of 
 $\mathfrak{g}$ which contains 
the ideal $[\mathfrak{g},\mathfrak{g}]$.

Now we equip a left invariant Riemannian metric
$g=\langle\cdot,\cdot\rangle$ on $G$. Denote by
$\mathfrak{u}$ the orthogonal 
complement of $\mathfrak{u}$ in $\mathfrak{g}$
with respect to $\langle\cdot,\cdot\rangle$. 
Then the homomorphism theorem implies that
$\dim \mathfrak{u}^{\perp}=\dim \mathfrak{g}/\mathfrak{u}\leq 1$.

The following criterion for unimodularity is known 
(see \cite[p.~317]{Milnor}).
\begin{Lemma}
A Lie group $G$ with a left invariant metric is unimodular if and only if 
$\mathfrak{u}=\mathfrak{g}$.
\end{Lemma}
Based on this criterion, the ideal $\mathfrak{u}$ is called the
\textit{unimodular kernel} of $\mathfrak{g}$. 
In particular, for a $3$-dimensional 
non-unimodular Lie group $G$, its unimodular kernel 
$\mathfrak{u}$ is commutative and of $2$-dimension.

\subsection{The unimodular basis}\label{sec:7.2}

Let $G$ be a $3$-dimensional unimodular
Lie group with a
left invariant Riemannian metric $g$. Then there exists an orthonormal basis
$\{e_1,e_2,e_3\}$ of $(\mathfrak{g},\langle 
\cdot,\cdot
\rangle)$ such that
\[
[e_2,e_3]=c_{1}e_{1},\ \ 
[e_3,e_1]=c_{2}e_{2},\ \ 
[e_1,e_2]=c_{3}e_{3}.
\] 
By the left translation of $G$,
$\mathcal{E}=\{e_1,e_2,e_3\}$
is regarded as a left invariant
orthonormal frame field of $G$.

According to the
signature of strcture constants,
the following classification table is
obtained by Milnor \cite{Milnor} (up to 
numeration of $c_1$, $c_2$ and $c_3$).

\begin{center}
\begin{tabular}{|c|c|c|}
\hline
Signature of $(c_1,c_2,c_3)$ & {} & Associated connected Lie group
\\
\hline
$(+,+,+)$  & simple, compact&
$\mathrm{SU}_2$ or 
$\mathrm{SO}_3$ 
\\
\hline
 $(+,+,-)$ or  $(+,-,-)$
 & simple, noncompact &
$\mathrm{SL}_{2}\mathbb{R}$ 
or $\mathrm{SO}^{+}_{2,1}$
\\
\hline
$(0,+,+)$
& solvable  &
$\mathrm{E}_{2}=\mathrm{SO}_2\ltimes\mathbb{R}^2$ 
\\
\hline
$(0,+,-)$
& solvable &
$\mathrm{E}_{1,1}=\mathrm{SO}^{+}_{1,1}\ltimes\mathbb{R}^2$ 
\\
\hline
$(+,0,0)$
& nilpotent &
Heisenberg group $\mathrm{Nil}_{3}$ 
\\
\hline
$(0,0,0)$
& Abelian &
$\mathbb{R}^{3}$ 
\\
\hline
\end{tabular}
\end{center}

The symmetric bilinear map $U$ is given by
\[
U(e_1,e_2)=\frac{1}{2}(-c_1+c_2)e_3,\ \
U(e_1,e_3)=\frac{1}{2}(c_1-c_3)e_2,\ \
U(e_2,e_3)=\frac{1}{2}(-c_2+c_3)e_1.
\]

The Levi-Civita connection of $g$ 
is described as 
\[
\nabla^{g}_{e_1}e_{1}=0,\  \nabla^{g}_{e_1}e_{2}=
\lambda_{1}e_{3},\ \
\nabla^{g}_{e_1}e_{3}=-\lambda_{1}e_{2},
\]
\[
\nabla^{g}_{e_2}e_{1}=-\lambda_{2}e_3,\  
\nabla^{g}_{e_2}e_{2}=0,\ \
\nabla^{g}_{e_2}e_{3}=\lambda_{2}e_{1},
\]
\[
\nabla^{g}_{e_3}e_{1}=\lambda_{3}e_{2},\ \
\nabla^{g}_{e_3}e_{2}=-\lambda_{3}e_{1},\ \
\nabla^{g}_{e_3}e_{3}=0,
\]

Here the constants $\{\lambda_{j}\}$ are defined by
\[
\lambda_{j}:=(c_1+c_2+c_3)/2-c_{j}, \quad 
j=1,2,3.
\]

The Riemannian curvature $R^g$ is given by
\[
R^g(e_1,e_2)e_1=(\lambda_{1}\lambda_{2}-c_{3}\lambda_{3})e_{2},\ \ 
R^g(e_1,e_2)e_2=-(\lambda_{1}\lambda_{2}-c_{3}\lambda_{3})e_{1},\ \ 
\]
\[
R^g(e_2,e_3)e_2=(\lambda_{2}\lambda_{3}-c_{1}\lambda_{1})e_{3},\ \ 
R^g(e_2,e_3)e_3=-(\lambda_{2}\lambda_{3}-c_{1}\lambda_{1})e_{2},\ \ 
\]
\[
R^g(e_1,e_3)e_1=(\lambda_{3}\lambda_{1}-c_{2}\lambda_{2})e_{3},\ \ 
R^g(e_1,e_3)e_3=-(\lambda_{3}\lambda_{1}-c_{2}\lambda_{2})e_{1}. 
\]

\begin{Example}\label{eg:SasakianG}
{\rm Every simply connected complete 
$3$-dimensional Sasakian $\varphi$-symmetric space 
is realized as a unimodular Lie group 
$G=G(c)$ whose Lie algebra 
$\mathfrak{g}(c)$ is generated by a 
unimodular basis 
$\{e_1,e_2,e_3\}$ 
satisfying
\[
[e_1,e_2]=2e_3,
\quad 
[e_2,e_3]=\frac{c+3}{2}e_1,
\quad 
[e_3,e_1]=\frac{c+3}{2}e_2
\]
to together with a 
left invariant Sasakian structure 
\[
\varphi e_1=e_2,\quad \varphi e_2=-e_1,
\quad 
\varphi e_3=0,
\quad 
\xi=e_3,\quad \eta=g(\xi,\cdot).
\]
One can see that 
\[
\mathfrak{g}(c)\cong 
\begin{cases}
\mathfrak{su}_2 \quad c>-3
\\
\mathfrak{nil}_3\quad c=-3
\\
\mathfrak{sl}_2\mathbb{R}\quad c<-3
\end{cases}
\]
In particular $G(c)$ is $\mathrm{SU}(2)=
\mathbb{S}^3$ equipped with a 
bi-invariant Riemannian metric of constant 
curvature $1$. 
As we saw in Proposition \ref{prop:2.5},
\[
\nabla^{(\alpha)}_{X}Y=\nabla^{g}_{X}Y-\frac{\alpha}{2}\eta(X)\eta(Y)\xi,
\quad \alpha\in\mathbb{R}
\]
gives a one-parameter family $\{(\varphi,\xi,\eta,g,\nabla^{(\alpha)})\}_{\alpha
\in\mathbb{R}}$ of left invariant Sasakian statistical 
structures on $G(c)$. 
In case $G(1)=\mathbb{S}^3$, $\nabla^{(\alpha)}$ coincides with 
the bi-invariant connection 
given in Corollary \ref{cor:5.5} with $K(X)Y=\alpha\eta(X)\eta(Y)\xi$.
}
\end{Example}

\subsection{Conjugate-symmetry}\label{sec:7.3}
We have characterized the $\alpha$-connections 
on the statistical manifold of the normal distribution 
as a left invariant connection compatible to the metric and 
its covariant derivative of the cubic form is totally symmetric.

In this subsection we consider conjugate-symmetry of $3$-dimensional 
unimodular statistical Lie groups.

\begin{Theorem}\label{thm:uni}
 Let $G$ be a $3$-dimensional unimodular Lie group equipped with 
 a left invariant statistical structure $(g,\nabla)$ such that $\nabla\not=\nabla^g$. 
If $(G,g,\nabla)$ is conjugate-symmetric then $G$ is locally
isomorphic to the special unitary group $\mathrm{SU}_2$, 
Euclidean motion group $\mathrm{E}_2
=\mathrm{SO}_2\ltimes\mathbb{R}^2$ or $\mathbb{R}^3$.
\end{Theorem}
\begin{proof}
Take a unimodular basis $\{e_1,e_2,e_3\}$ and 
express the skewness operator $K$ as 
\[
K(e_i)e_j=\sum_{\ell=1}^{3}K_{ij}^{\ell}e_{\ell}, 
\quad i,j=1,2,3.
\]
Then we know that 
\[
K_{ij}^{\ell}=K_{ji}^{\ell}, \quad 
K_{ij}^{\ell}=K_{i\ell}^{j}.
\]
The total symmetry 
of $\nabla^{g}K$ is 
the system
\begin{equation}\label{eq:a-System}
K^{1}_{23}(\lambda_1+\lambda_2)=
K^{1}_{23}(\lambda_2+\lambda_3)=
K^{1}_{23}(\lambda_3+\lambda_1)=0,
\end{equation} 
\begin{equation}\label{eq:b-System}
K^{1}_{31}(\lambda_1+3\lambda_2)=
K^{2}_{12}(\lambda_2+3\lambda_3)=
K^{3}_{23}(\lambda_3+3\lambda_1)=0,
\end{equation}
\begin{equation}\label{eq:c-System}
K^{1}_{12}(\lambda_1+3\lambda_3)=
K^{2}_{23}(\lambda_2+3\lambda_1)=
K^{3}_{31}(\lambda_3+3\lambda_2)=0,
\end{equation}
\begin{align}\label{eq:d2}
&\lambda_{1}K^{2}_{22}+\lambda_{2}K^{1}_{12}
-(2\lambda_1+\lambda_2)K^{2}_{33}=0,
\\
& \lambda_{2}K^{2}_{33}+\lambda_{3}K^{3}_{22}
-(2\lambda_2+\lambda_3)K^{3}_{11}=0,
\label{eq:d3}
\\
& \lambda_{1}K^{1}_{33}+\lambda_{3}K^{1}_{11}
-(2\lambda_3+\lambda_1)K^{1}_{22}=0,
\label{eq:d1}
\end{align}
\begin{align}\label{eq:e3}
&
\lambda_{1}K^{3}_{33}+\lambda_{3}K^{3}_{11}
-(2\lambda_1+\lambda_3)K^{3}_{22}=0,
\\
& 
\lambda_{1}K^{1}_{22}+\lambda_{2}K^{1}_{11}
-(2\lambda_2+\lambda_1)K^{1}_{33}=0,
\label{eq:e1}
\\
& \lambda_{2}K^{2}_{33}+\lambda_{3}K^{2}_{22}
-(2\lambda_3+\lambda_2)K^{2}_{11}=0.
\label{eq:e2}
\end{align}

First we consider the case $\lambda_1=\lambda_2=\lambda_3=0$. 
In this case $\mathfrak{g}$ is abelian. For any tensor $K$
of type $(1,2)$ on $\mathfrak{g}$, 
$\nabla:=\nabla^{g}-K/2$ gives a conjugate 
symmetric left invariant statistical structure. 

Hereafter we assume that $\lambda_1^2+\lambda_2^2+\lambda_3^2\not=0$. 
Under this assumption we have 
$(\lambda_1+\lambda_2)^{2}+(\lambda_2+\lambda_3)^{2}+(\lambda_3+\lambda_1)^{2}\not=0$. 
Hence the system \eqref{eq:a-System} implies that $K_{12}^{3}=0$.

\begin{enumerate}
\item The case: All of $\lambda_1+3\lambda_2$, $\lambda_2+3\lambda_3$ and 
$\lambda_3+3\lambda_1$ are non-zero:

In this case \eqref{eq:b-System} implies that 
\[
K^{1}_{13}=K^{2}_{12}=K^{3}_{23}=0.
\]
Since $K_{12}^{2}=0$, the equations 
\eqref{eq:d1} and \eqref{eq:e1} are reduced to 
\[
\lambda_{2}K^{1}_{11}-(\lambda_1+2\lambda_2)K^{1}_{33}=0,
\quad 
\lambda_{3}K^{1}_{11}+\lambda_{1}K^{1}_{33}=0.
\]
We dedide our discussion into 
two parts: (i) $3\lambda_2+\lambda_3\not=0$ and 
(ii) $3\lambda_2+\lambda_3=0$. 

In the former case, we notice that $K^{1}_{33}$ from 
\eqref{eq:c-System}. The equations 
\eqref{eq:d1} and \eqref{eq:e1} are reduced more to 
\[
\lambda_{2}K^{1}_{11}=\lambda_{3}K^{1}_{11}=0.
\]
If $K^{1}_{11}\not=0$, we have $\lambda_2=\lambda_3=0$. This 
contradicts to the assumption 
$\lambda_2+3\lambda_3\not=0$. Hence $K^{1}_{11}=0$.

In the latter case, the equations 
\eqref{eq:d1} and \eqref{eq:e1} imply that $K^{1}_{33}=0$. Hence we obtain 
$K^{1}_{11}=0$ again. By analogous discussions on 
\eqref{eq:d2}-\eqref{eq:e2} and \eqref{eq:d3}-\eqref{eq:e3}, 
we deduce that $K^{\ell}_{ij}=0$ for all $i$, $j$ and $k$.

It should be remarked that one can deduce that $K=0$ when 
all of 
$3\lambda_1+\lambda_2$, 
$3\lambda_2+\lambda_3$,
$3\lambda_3+\lambda_1$ are
non-zero. 

\item The case: Two of $\lambda_1+3\lambda_2$, $\lambda_2+3\lambda_3$ and 
$\lambda_3+3\lambda_1$ are non-zero and the remain one is non-zero:

By numeration, it suffices to consider the case
\[
\lambda_1+2\lambda_2=0,
\quad 
\lambda_2+2\lambda_3=0,
\quad 
\lambda_3+2\lambda_1\not=0.
\quad 
\]
In this case, 
\[
3\lambda_3+\lambda_1=12\lambda_3\not=0,
\quad
3\lambda_1+\lambda_2=24\lambda_3\not=0,
\quad
3\lambda_2+\lambda_3=-8\lambda_3\not=0.
\]
As we have pointed our above, $K=0$ holds under this condition.

\item The case: Only one of $\lambda_1+3\lambda_2$, $\lambda_2+3\lambda_3$ and 
$\lambda_3+3\lambda_1$ is zero and the other two are non-zero:

By numeration, it suffices to consider the case
\[
\lambda_1+3\lambda_2=0,
\quad 
\lambda_2+3\lambda_3\not=0,
\quad 
\lambda_3+3\lambda_1\not=0.
\quad 
\]
Note that the first equation is equivalent to the relation $c_1-c_2+2c_3=0$.
\begin{enumerate}
\item Subcase: $3\lambda_1+\lambda_2\not=0$, $3\lambda_2+\lambda_3=0$ and $3\lambda_3+\lambda_1\not=0$:
Since $3\lambda_2+\lambda_3=0$ is equivalent to to $2c_1-c_2+c_3=0$, 
the structure constants $\{c_1,c_2,c_3\}$ satisfy 
$c_1:c_2:c_3=1:3:1$. Hence the Lie algebra is $\mathfrak{su}(2)$. The 
left invariant metric has $4$-dimensional isometry group (and it is 
naturally reductive). Note that $\lambda_1=\lambda_3=-3\lambda_2\not=0$. 

From \eqref{eq:a-System} and 
\eqref{eq:b-System} we get
\[
K^{1}_{22}=K^{2}_{33}=K^{2}_{23}=K^{1}_{12}=0.
\]
Next, from \eqref{eq:d1}-\eqref{eq:e1}, 
we have 
$K^{1}_{11}+K^{1}_{33}=0$. 
On the other hand, 
 \eqref{eq:d2}--\eqref{eq:e2} imply that
 $K^{2}_{22}=K^{1}_{12}=0$. 
 Finally from \eqref{eq:d3}--\eqref{eq:e3}, 
 we have
 $K^{3}_{33}+K^{1}_{13}=0$. 
 
The Lie group $G$ admits non-trivial 
conjugate symmetric left invariant 
statistical structure whose skewness field 
is parametrized as
\[
K^{1}_{11}=-K^{1}_{33}=\alpha, 
\quad 
K^{3}_{33}=-K^{1}_{13}=\beta,
\quad 
\alpha,\beta\in\mathbb{R}.
\]
The skewness operator $K=K^{(\alpha,\beta)}$
has the form 
\[
K^{(\alpha,\beta)}=\alpha K^{(1,0)}+\beta K^{(0,1)}.
\]
Thus $\nabla^{g}-K^{(\alpha,0)}/2$ is the 
$\alpha$-connection of $\nabla^{g}-K^{(1,0)}/2$. 
Moreover, $\nabla^{g}-K^{(0,\alpha)}/2$ is the 
$\alpha$-connection of $\nabla^{g}-K^{(0,1)}/2$. 

\item Subcase: $3\lambda_1+\lambda_2\not=0$, $3\lambda_2+\lambda_3\not=0$ and $3\lambda_3+\lambda_1=0$:
Since $3\lambda_3+\lambda_1=0$ is equivalent to to $c_1+2c_2-c_3=0$, 
the structure constants $\{c_1,c_2,c_3\}$ satisfy 
$c_1:c_2:c_3=1:-1:-1$. Hence the Lie algebra is $\mathfrak{sl}_2\mathbb{R}$.

From \eqref{eq:b-System} and 
\eqref{eq:c-System} we get
\[
K^{1}_{22}=K^{2}_{33}=K^{2}_{23}=K^{1}_{33}=0.
\]
Next, from \eqref{eq:d1}-\eqref{eq:e1}, 
we have 
$K^{1}_{11}+K^{1}_{22}=0$. 
On the other hand, 
 \eqref{eq:d2}--\eqref{eq:e2} imply that
 $K^{2}_{22}=K^{1}_{12}=0$. 
 Finally from \eqref{eq:d3}--\eqref{eq:e3}, 
 we have
 $K^{3}_{33}+K^{1}_{13}=0$. 
 Hence $K=0$. 
 
\item Subcase: $3\lambda_1+\lambda_2=0$, 
$3\lambda_2+\lambda_3\not=0$ and $3\lambda_3+\lambda_1\not=0$:
In this case we have $\lambda_1+\lambda_2=0$. This relation 
implies that $c_1=c_2$ and $c_3=0$ and hence 
$\lambda_1=\lambda_2=0$. Since we assumed that 
$\mathfrak{g}$ is non-abelian, $\lambda_3\not=0$.
Moreover the Lie algebra $\mathfrak{g}$ is 
$\mathfrak{e}(2)$.

From \eqref{eq:b-System} and 
\eqref{eq:c-System} we get
\[
K^{1}_{22}=K^{2}_{33}=K^{1}_{12}=K^{1}_{33}=0.
\]
Next, from \eqref{eq:d1}-\eqref{eq:e1}, 
we have 
$K^{1}_{11}=0$. 
On the other hand, 
 \eqref{eq:d2}--\eqref{eq:e2} imply that
 $K^{2}_{22}=0$. 
 Finally from \eqref{eq:d3}--\eqref{eq:e3}, 
 we have
 $K^{1}_{13}=K^{2}_{23}=0$. 
 
The Lie group $G$ admits non-trivial 
conjugate symmetric left invariant 
statistical structure whose skewness field 
is parametrized as
\[
K^{1}_{13}=K^{2}_{23}=\alpha, 
\quad 
K^{3}_{33}=\beta,
\quad 
\alpha,\beta\in\mathbb{R}.
\]
The skewness operator $K=K^{(\alpha,\beta)}$
has the form 
\[
K^{(\alpha,\beta)}=\alpha K^{(1,0)}+\beta K^{(0,1)}.
\]
Thus $\nabla^{g}-K^{(\alpha,0)}/2$ is the 
$\alpha$-connection of $\nabla^{g}-K^{(1,0)}/2$. 
Moreover, $\nabla^{g}-K^{(0,\alpha)}/2$ is the 
$\alpha$-connection of $\nabla^{g}-K^{(0,1)}/2$. 

\end{enumerate}
\end{enumerate}
\end{proof}

\begin{Remark}{\rm
Let $G(c)$ be a simply connected 
unimodular Lie group equipped with a 
left invariant Sasakian $\varphi$-symmetric 
structure exhibited in Example 
\ref{eg:SasakianG}. The $\alpha$-connection is 
given by
\[
\nabla^{(\alpha)}_{X}Y=
\nabla^{g}_{X}Y-\frac{\alpha}{2}K(X)Y, 
\quad 
K(X)Y=\eta(X)\eta(Y)\xi.
\]
The skeness operator has only one 
non-zero component $K^{3}_{33}=1$.  
The statistical structure 
$(g,\nabla^{(1)})$ is 
conjugate-symmetric if and only if 
$\lambda_1=0$. This 
is equivalent to $-c_1+c_2+c_3$. 
On the other hand, 
$-c_1+c_2+c_3=2$. Thus 
$(G(c),g,\nabla^{(1)})$ is not conjugate symmetric.
}
\end{Remark}


\section{Three dimensional non-unimodular Lie groups}\label{sec:8}
\subsection{Normalization}\label{sec:8.1}
Let $G$ be a non-unimodular $3$-dimensional Lie group with a 
left invariant metric $g$ and with  Lie algebra 
$\mathfrak{g}$.
Then the non-unimodular property of $G$
implies that the unimodular kernel 
$\mathfrak{u}$ is a $2$-dimensional ideal of $\mathfrak{g}$ which contains 
the ideal $[\mathfrak{g},\mathfrak{g}]$. 

On the Lie algebra $\mathfrak{g}$, we can take an orthonormal basis
$\{e_1,e_2,e_3\}$ such that
\begin{enumerate}
\item $\langle e_{1}, X\rangle=0\ \mbox{for all}\> X \in \mathfrak{u}$,
\item $\langle [e_1,e_2], [e_1,e_3]\rangle=0$,
\end{enumerate}
where $\{e_2,e_3\}$ is an orthonormal basis of $\mathfrak{u}$.

The representation matrix $A$ of $\mathrm{ad}(e_1):\mathfrak{u}
\to\mathfrak{u}$ relative to the basis $\{e_2,e_3\}$ is expressed as
\[
A=\left(
\begin{array}{cc}
\alpha & \gamma
\\
\beta & \delta
\end{array}
\right).
\]
Then the commutation relations of the basis are given by
\[
[e_1,e_2]=\alpha e_{2}+\beta e_{3},\quad 
[e_2,e_3]=0,\quad  
[e_1,e_3]=\gamma e_2+\delta e_3,
\]
with $\mathrm{tr}\>A=\alpha+\delta\not=0$ and $\alpha\gamma+\beta \delta=0$.
These commutation relations imply that $\mathfrak{g}$ is solvable.

Under a suitable homothetic change of the metric, we may assume that
$\alpha+\delta=2$. Then the constants $\alpha$,
$\beta$, $\gamma$ and $\delta$ are represented
as 
\[
\alpha=1+\xi,\quad  \beta=(1+\xi)\eta,\quad 
\gamma=-(1-\xi)\eta,\quad 
\delta=1-\xi.
\]
If necessarily, by changing the sign of $e_1$, $e_2$ and $e_3$,
we may assume that the constants $\xi$ and $\eta$ satisfy the condition $\xi,\eta\geq 0$. 
We note that for the case that 
$\xi=\eta=0$, $(G,g)$ has constant negative curvature 
(see Example \ref{Example2.2}).

From now on we work under this normalization.
Then the commutation relations are rewritten as 
\begin{equation}\label{eq:nonUnibasis}
[e_1,e_2]=(1+\xi)(e_{2}+\eta e_{3}),
\quad 
[e_2,e_3]=0,
\quad 
[e_3,e_1]=(1-\xi)(\eta e_{2}-e_{3}).
\end{equation}
We refer 
$(\xi,\eta)$ as the \emph{structure constants} 
of the non-unimodular Lie algebra $\mathfrak{g}$.

Non-unimodular Lie algebras $\mathfrak{g}=\mathfrak{g}(\xi,\eta)$ are classified by the 
\emph{Milnor invariant} $\mathcal{D}:=\det A=(1-\xi^2)(1+\eta^2)$.
More precisely we have the following fact (\textit{cf} \cite{Milnor}).
\begin{Proposition}
For any pair of $(\xi,\eta)$ and 
$(\xi^\prime,\eta^\prime)$ which are not $(0,0)$, two Lie algebras 
$\mathfrak{g}(\xi,\eta)$ and $\mathfrak{g}(\xi^\prime,\eta^\prime)$ are 
isomorphic if and only if their Milnor invariants $\mathcal{D}$ and 
$\mathcal{D}^\prime$ agree.
\end{Proposition}
The characteristic equation for the matrix $A$ is $\lambda^2-2\lambda+\mathcal{D}=0$. 
Thus the eigenvalues of $A$ are $1\pm\sqrt{1-\mathcal{D}}$. 
Roughly speaking, this means that the moduli space of 
non-unimodular Lie algebras $\mathfrak{g}(\xi,\eta)$ may be 
decomposed into three kinds of types which are determined by the 
conditions $\mathcal{D}>1$, $\mathcal{D}=1$ and $\mathcal{D}<1$.

The Levi-Civita connection of $(G,g)$ is given by the following table:
\begin{Proposition}\label{Proposition2.1}
\[
\begin{array}{lll}
\nabla^{g}_{e_1}e_{1}=0, & \nabla^{g}_{e_1}e_{2}=\eta e_{3}, & 
\nabla^{g}_{e_1}e_{3}=-\eta e_{2}\\
\nabla^{g}_{e_2}e_{1}=-(1+\xi)e_{2}-\xi\eta e_{3}, 
& \nabla^{g}_{e_2}e_{2}=(1+\xi) e_1, & 
\nabla^{g}_{e_2}e_{3}=\xi\eta e_{1}\\
\nabla^{g}_{e_3}e_{1}=-\xi\eta e_{2}
-(1-\xi)e_{3}, 
& \nabla^{g}_{e_3}e_{2}=\xi\eta 
e_{1} & \nabla^{g}_{e_3}e_{3}=(1-\xi)e_1.
\end{array}
\]
The Riemannian curvature $R$ is given by
\begin{align*}
R^{g}(e_1,e_2)e_1=&
\{\xi\eta^{2}+(1+\xi)^{2}+\xi\eta^{2}(1+\xi)\}e_{2},
\\
R^{g}(e_1,e_2)e_2=&-\{\xi\eta^{2}+(1+\xi)^{2}+\xi\eta^{2}(1+\xi)\}e_{1},
\\
R^{g}(e_1,e_3)e_1=&-\{\xi\eta^{2}-(1-\xi)^{2}+\xi\eta^{2}(1-\xi)\}e_{3},
\\
R^{g}(e_1,e_3)e_3=&\{\xi\eta^{2}-(1-\xi)^{2}+\xi\eta^{2}(1-\xi)\}e_{1},
\\
R^{g}(e_2,e_3)e_2=&\{1-\xi^{2}(1+\eta^{2})\}e_{3},
\\
R^{g}(e_2,e_3)e_3=&-\{1-\xi^{2}(1+\eta^{2})\}e_{2}
\end{align*}
and $R^{g}(e_i,e_j)e_k=0$ if otherwise.

The basis $\{e_1,e_2,e_3\}$ diagonalizes the 
Ricci tensor field $\mathrm{Ric}$.
The principal Ricci curvatures are given by
\begin{align}
\mathrm{Ric}^{g}(e_{1},e_{1})&=-2\{1+\xi^{2}(1+\eta^{2})\}<-2,
\nonumber
\\
\mathrm{Ric}^{g}(e_{2},e_{2})&=-2\{1+\xi(1+\eta^{2})\}<-2,
\\
\mathrm{Ric}^{g}(e_{3},e_{3})
&=-2\{1-\xi(1+\eta^{2})\}.
\nonumber
\end{align}
The scalar curvature $\rho^{g}$ is 
\[
\rho^{g}=-2\{3+\xi^{2}(1+\eta^2)\}<0.
\]
\end{Proposition}
\begin{Remark}
{\rm The sign of $\mathrm{Ric}^{g}(e_{3},e_{3})$ can be positive, null or negative.
In fact, if $\mathcal{D}<0$, then $\mathrm{Ric}^{g}(e_{3},e_{3})>0$. 
If $\mathcal{D}\geq 0$, then $\mathrm{Ric}^{g}(e_{3},e_{3})
=0$ is possible. In case $\mathcal{D}>0$, there exists a 
metric of strictly negative curvature. Moreover if 
$\mathcal{D}>1$, there exists a metric of constant negative curvature.
}
\end{Remark}


\subsection{Explicit models}\label{sec:8.2}
Here we give explicit matrix group models of
non-unimodular Lie groups.  

The simply connected Lie 
group $\widetilde{G}$ corresponding to 
the non-unimodular Lie algebra $\mathfrak{g}$ with 
structure constants $(\xi,\eta)$
is given explicitly by
\[
\widetilde{G}(\xi,\eta)=
\left\{
\left.
\left(
\begin{array}{cccc}
1 & 0 & 0 & x\\
0 & \alpha_{11}(x) & \alpha_{12}(x) & y\\
0 & \alpha_{21}(x) & \alpha_{22}(x) & z\\
0 & 0 & 0 &1
\end{array}
\right)
\
\right
\vert
\
x,y,z \in \mathbb{R}
\right\},
\]
where $\alpha_{ij}(x)$ is the $(i,j)$-entry of 
$\exp (x A)$. This shows that 
$\widetilde{G}(\xi,\eta)$ is the semi-direct product
$\mathbb{R}\ltimes\mathbb{R}^2$ 
with multiplication
\begin{equation}\label{semi-direct}
(x,y,z)\cdot (x^\prime,y^\prime,z^\prime)=
(x+x^\prime,y+\alpha_{11}(x)y^\prime+\alpha_{12}(x)z^\prime,
z+\alpha_{21}(x)y^\prime+\alpha_{22}(x)z^\prime).
\end{equation}
The Lie algebra of $\widetilde{G}(\xi,\eta)$ is spanned by the basis
\[
e_{1}=
\left(
\begin{array}{cccc}
0 & 0 & 0 & 1\\
0 & 1+\xi & -(1-\xi)\eta & 0\\
0 & (1+\xi)\eta & 1-\xi & 0\\
0 & 0 & 0 &0
\end{array}
\right),
\ \
e_{2}=
\left(
\begin{array}{cccc}
0 & 0 & 0 & 0\\
0 & 0 & 0 & 1\\
0 & 0 & 0 & 0\\
0 & 0 & 0 &0
\end{array}
\right),
\ \
e_{3}=
\left(
\begin{array}{cccc}
0 & 0 & 0 & 0\\
0 & 0 & 0 & 0\\
0 & 0 & 0 & 1\\
0 & 0 & 0 &0
\end{array}
\right).
\]
This basis satisfies the commutation relations
\[
[e_1,e_2]=(1+\xi)\{e_{2}+\eta e_{3}\},\ \
[e_2,e_3]=0,\ \ 
[e_3,e_1]=(1-\xi)\{\eta e_2-e_{3}\}.
\]
Thus the Lie algebra of 
$\widetilde{G}(\xi,\eta)$
is the non-unimodular Lie algebra 
$\mathfrak{g}=\mathfrak{g}(\xi,\eta)$ with structure constants $(\xi,\eta)$. 
The left invariant vector fields corresponding to 
$e_1$, $e_2$ and $e_3$ are
\[
e_1=\frac{\partial}{\partial x},\ \ 
e_2=\alpha_{11}(x)\frac{\partial}{\partial y}+
\alpha_{21}(x)\frac{\partial}{\partial z},
\ \
e_3=\alpha_{12}(x)\frac{\partial}{\partial y}
+\alpha_{22}(x)\frac{\partial}{\partial z}.
\]
\begin{Example}[$\xi=0$, $\mathcal{D}\geq 1$]\label{Example2.2}{\rm 
The simply connected Lie group 
$\widetilde{G}(0,\eta)$ is isometric to the hyperbolic $3$-space 
$\mathbb{H}^3=\mathbb{H}^3(-1)$ of curvature $-1$
and given 
explicitly by
\[
\widetilde{G}(0,\eta)=
\left\{
\left(
\begin{array}{cccc}
1 & 0 & 0 & x\\
0 & e^{x}\cos(\eta x) & -e^{x}\sin(\eta x) & y\\
0 & e^{x}\sin(\eta x) & e^{x}\cos(\eta x) & z\\
0 & 0 & 0 &1
\end{array}
\right)
\
\biggr
\vert
\
x,y,z \in \mathbb{R}
\right\}. 
\]
The left invariant metric is 
$dx^2+e^{-2x}(dy^2+dz^2)$. 
Thus $\widetilde{G}(0,\eta)$
is the warped product model of $\mathbb{H}^3$.
In fact, put $w=e^{x}$ then the left invariant metric of 
$\widetilde{G}(0,\eta)$ 
is rewritten as 
the Poincar{\'e} metric
\[
\frac{dy^2+dz^2+dw^2}{w^2}.
\]
The Milnor 
invariant of 
$\widetilde{G}(0,\eta)$ is $\mathcal{D}=1+\eta^2\geq 1$.
}
\end{Example}

\begin{Example}[$\eta=0$, 
$\mathcal{D}\leq1$]\label{Example2.3}{\rm
For each $\xi\geq 0$, $\widetilde{G}(\xi,0)$ is given by:
\[
\widetilde{G}(\xi,0)=\left\{
\left(
\begin{array}{cccc}
1 & 0 & 0 & x\\
0 & e^{(1+\xi)x} & 0 & y\\
0 & 0 & e^{(1-\xi)x} & z\\
0 & 0 & 0 &1
\end{array}
\right)
\
\biggr
\vert
\
x,y,z \in \mathbb{R}
\right\}.
\]
The left invariant Riemannian metric is given explicitly by
\[
dx^{2}+e^{-2(1+\xi)x}dy^{2}+e^{-2(1-\xi)x}dz^{2}.
\]
The Milnor invariant is $\mathcal{D}=1-\xi^2\leq 1$.

Here we observe locally symmetric examples:
\begin{itemize}
\item If $\xi=0$ then $\widetilde{G}(0,0)$ is a warped product model of 
hyperbolic 3-space $\mathbb{H}^3(-1)$. 
\item If $\xi=1$ then $\widetilde{G}(1,0)$ is isometric to the Riemannian product 
$\mathbb{H}^{2}(-4)\times \mathbb{R}$ of hyperbolic 
plane $\mathbb{H}^{2}(-4)$ of curvature $-4$ and the 
real line.
In fact, via the coordinate change $(u,v)=(2y,e^{2x})$, the metric is rewritten as
\[
g=\frac{du^2+dv^2}{4v^2}+dz^2.
\]
\end{itemize}
}
\end{Example}

\begin{Example}[$\xi=1$, 
$\mathcal{D}=0$]\label{Example2.4}{\rm
Assume that $\xi=1$. Then $\widetilde{G}(1,\eta)$ is given explicitly by
\[
\widetilde{G}(1,\eta)=
\left\{
\left.
\left(
\begin{array}{cccc}
1 & 0 & 0 & x\\
0 & e^{2x} & 0 & y\\
0 & \eta(e^{2x}-1) & 1 & z\\
0 & 0 & 0 &1
\end{array}
\right)
\
\right
\vert
\
x,y,z \in \mathbb{R}
\right\}.
\]
The left invariant metric is 
\[
dx^2+\{e^{-4x}+\eta^2(1-e^{-2x})^2\}dy^2-2\eta(1-e^{-2x})dydz+dz^2.
\]
The non-unimodular Lie group $\widetilde{G}(1,\eta)$ has sectional curvatures
\[
K^{g}_{12}=-3\eta^{2}-4,
\quad 
K^{g}_{13}=K^{g}_{23}=\eta^{2},
\]
where $K^{g}_{ij}$ ($i\not=j$) denote the 
sectional curvatures of the planes spanned
by vectors $e_i$ and $e_j$. 
One can check that $\widetilde{G}(1,\eta)$ is isometric to
the so-called 
\emph{Bianchi-Cartan-Vranceanu space} $M^{3}(-4,\eta)$
with $4$-dimensional isometry group and isotropy subgroup $\mathrm{SO}_2$:
\[
M^{3}(-4,\eta)=
\left(
\{(x,y,z)\in \mathbb{R}^{3}
\
\vert 
\
x^2+y^2<1\},ds^2
\right),
\]
\[
ds^2=\frac{dx^2+dy^2}{(1-x^2-y^2)^{2}}+
\left(
dz+ \frac{\eta(ydx-xdy)}{1-x^2-y^2}
\right)^{2}
\]
with $\eta\geq 0$.  
The family $\{\widetilde{G}(1,\eta)\}_{\eta\geq 0}$ is 
characterized by the condition $\mathcal{D}=0$. 
In particular $M^{3}(-4,\eta)$ with \textit{positive} $\eta$ is isometric to 
the universal covering $\widetilde{\mathrm{SL}}_2\mathbb{R}$
of the special linear group equipped with the above metric 
but 
\emph{not
isomorphic} to $\widetilde{\mathrm{SL}}_2\mathbb{R}$
as Lie groups. We here note that $\mathrm{SL}_2\mathbb{R}$ is a 
\emph{unimodular} Lie group, while $\widetilde{G}(1,\eta)$ is
\emph{non-unimodular}. 
}
\end{Example}
For more informations on left invariant metrics on 
3-dimensional non-unimodular Lie groups, we refer to
\cite{IN}.

\begin{Example}
{\rm 
Here we consider the product group 
$G(\nu_1,\nu_2)\times\mathbb{R}$ 
of the $2$-dimensional solvable Lie group 
$G(\nu_1,\nu_2)$ studied in Section 
\ref{sec:6.3} and the abelian group $\mathbb{R}=(\mathbb{R},+)$.
The product Lie group 
is isomorphic to
\[
L=\left\{
\left.
\left(
\begin{array}{ccc}
y & x &0
\\ 
0 & 1 & 0
\\ 0 & 0 & e^{z}
\end{array}
\right)
\>\right\vert
\>
x,y,z\in\mathbb{R},\>y>0
\right\}.
\]
The product metric
\[
g=\frac{\nu_1^2dx^2+\nu_2^2dy^2}{y^2}+dz^2
\]
is left invariant on $L$. 
The Lie algebra 
\[
\mathfrak{l}=\left\{
\left.
\left(
\begin{array}{ccc}
v/\nu_2 & u/\nu_1 &0
\\ 
0 & 0 & 0
\\ 0 & 0 & w
\end{array}
\right)
\>\right\vert
\>
u,v,w\in\mathbb{R}
\right\}
\]
of $L$ is spanned by the orthonormal basis
\[
e_1
=
\left(
\begin{array}{ccc}
1/\sqrt{\nu_2} & 0 &0
\\ 
0 & 0 & 0
\\ 0 & 0 & 0
\end{array}
\right),
\quad 
e_2=\left(
\begin{array}{ccc}
0 & 0 &0
\\ 
0 & 0 & 0
\\ 0 & 0 & 1
\end{array}
\right),
\quad 
e_3=
\left(
\begin{array}{ccc}
0 & 1/\sqrt{\nu_1} &0
\\ 
0 & 0 & 0
\\ 0 & 0 &0
\end{array}
\right)
\]
with commutation relations
\[
[e_1,e_2]=0, 
\quad 
[e_2,e_3]=0,
\quad
[e_3,e_1]=-\frac{1}{\nu_2}e_3.
\]
This shows that $L$ is non-unimodular. 
Indeed, the unimodular kernel is spanned by
$e_1$ and $e_2$. The representation matrix $A$ of $\mathrm{ad}(e_3)$ 
relative to $\{e_1,e_2\}$ has 
$\mathrm{tr}\>A=1/\nu_2$ and 
$\det A=0$. 
Thus $\mathrm{tr}\>A=2$ when and only when $\nu_2=1/2$.

The Levi-Civita connection 
is described as
\[
\nabla^{g}_{e_1}e_1=
\nabla^{g}_{e_1}e_2=
\nabla^{g}_{e_1}e_3=0,
\quad 
\nabla^{g}_{e_2}e_1=
\nabla^{g}_{e_2}e_2=
\nabla^{g}_{e_2}e_3=0,
\]
\[
\nabla^{g}_{e_3}e_1=-\frac{1}{\nu_2}e_3,
\quad 
\nabla^{g}_{e_3}e_2=0,
\quad 
\nabla^{g}_{e_3}e_3=\frac{1}{\nu_2}e_1.
\]
Let us classify the left invariant conjugate symmetric 
statistical structures associated the product metric $g$. 
The total symmetry of $\nabla^{g}K$ is the systems:
\begin{align}
    &-K_{11}^{1}+2K_{33}^{1}=0,  \label{eq:sol1} \\
    &-K_{12}^{1}+K_{33}^{2}=0,  \label{eq:sol2} \\
    &-2K_{13}^{1}+K_{33}^{3}=0,  \label{eq:sol3} \\
    &K_{13}^{1}=K_{22}^{1}=K_{23}^{1}=K_{23}^{2}=0. \label{eq:sol4}
\end{align}
From \eqref{eq:sol3} and \eqref{eq:sol4} we get $K_{33}^3=0$. 
The product group $G(\nu_1,\nu_2)\times\mathbb{R}$ 
admits non-trivial conjugate symmetric left invariant statistical structure whose skewness field is parametrized as
\begin{align*}
K_{11}^{1}=2K_{33}^1=2\sqrt{2}\alpha,\quad K_{12}^1=K_{33}^{2}=\beta,\quad 
K_{22}^2=\gamma,
\quad \alpha,\beta,\gamma\in\mathbb{R}.
\end{align*}
Then the skewness operator $K=K^{(\alpha,\beta,\gamma)}$ has the form 
$K^{(\alpha,\beta,\gamma)}=\alpha K^{(1,0,0)}+\beta K^{(0,1,0)}+\gamma K^{(0,0,1)}$. 
One can see that $K^{(\alpha,0,0)}$ is the $\alpha$-connection of $K^{(1,0,0)}$. 
Moreover the restriction of  $K^{(\alpha,0,0)}$ to $\mathfrak{g}(\nu_1,\nu_2)$ coincides 
with the skewness operator of the $\alpha$-connection of $G(\nu_1,\nu_2)$. 
As a result, the product Lie group 
$G(1,\sqrt{2})\times\mathbb{R}$ of the statistical Lie group 
of the normal distribution and the real line $(\mathbb{R},+)$ admits 
equipped with the product metric admits left invariant 
compatible connections which are conjugate symmetric. 
Note that $G(1,\sqrt{2})\times\mathbb{R}$ does not satisfy the 
normalization $\mathrm{tr}\>A=2$ (indeed, $\mathrm{tr}\>A=1/\sqrt{2}$). 
}
\end{Example}

\subsection{Conjugate-Symmetry}\label{sec:8.3}
In this subsection we consider conjugate-symmetry $3$-dimensional non-unimodular statistical Lie groups.

\begin{Theorem}\label{thm:nonuni}
    Let $G$ be a $3$-dimensional non-unimodular Lie group equipped with a 
	left invariant statistical structure $(g,\nabla)$ such that 
	$\nabla\ne\nabla^g$. If $(G,g,\nabla)$ is conjugate-symmetric then $G$ 
	is locally isomorphic to the solvable Lie group model 
    of the hyperbolic $3$-space or 
    the product Lie group $\mathbb{H}^2(-4)\times\mathbb{R}$ of 
    the solvable Lie group model of the hyperbolic plane and 
    the real line.
\end{Theorem}
\begin{proof}
Take the basis $\{e_1,e_2,e_3\}$ of the Lie algebra 
$\mathfrak{g}$ of $G$ satisfying \eqref{eq:nonUnibasis} 
and express the skewness operator $K$ as 
\[
K(e_i)e_j=\sum_{\ell=1}^{3}K_{ij}^{\ell}e_{\ell}, 
\quad i,j=1,2,3.
\]
Then the total symmetry of $\nabla^gK$ is the system
    
    \begin{align}
        -(1+\xi)K_{11}^{1}+2(1+\xi)K_{22}^{1}+2\eta(1+\xi) K_{23}^{1}&=0, \label{a2} \\
        -(1-\xi)K_{11}^{1}-2\eta(1-\xi)K_{23}^{1}+2(1-\xi)K_{33}^{1}&=0, \label{c3}
    \end{align}
    \begin{align}
        -\xi\eta K_{11}^{1}-\eta K_{22}^{1}+2(1+\xi)K_{23}^{1}+\eta(1+2\xi)K_{33}^{1}&=0, \label{a3} \\
        -\xi\eta K_{11}^{1}+\eta(2\xi-1)K_{22}^{1}+2(1-\xi)K_{23}^{1}+\eta K_{33}^{1}&=0, \label{c2}
    \end{align}
    \begin{align}
        (1+\xi)K_{23}^{1}&=\xi \eta K_{22}^{1}, \label{e2} \\
        (1-\xi)K_{22}^{1}&=(1+\xi)K_{33}^{1}, \label{e3} \\
        (1-\xi)K_{23}^{1}&=\xi\eta K_{33}^{1}. \label{f3}
    \end{align}
    \begin{align}
        3(1+\xi)K_{12}^{1}+\eta(1+3\xi)K_{13}^{1}&=0,\label{a1} \\
        \eta(3\xi-1)K_{12}^{1}+3(1-\xi)K_{13}^{1}&=0, \label{c1}
    \end{align}
    \begin{align}
        -2(1+\xi)K_{12}^{1}+(1+\xi)K_{22}^{2}+\eta(3+\xi)K_{23}^{2}&=0,\label{b2} \\
        -2(1-\xi)K_{13}^{1}-\eta(3-\xi)K_{33}^{2}+(1-\xi)K_{33}^{3}&=0,\label{d3}
    \end{align}
    \begin{align}
        -\xi\eta K_{12}^{1}-(1+\xi)K_{13}^{1}-\eta K_{22}^{2}+(1+\xi)K_{23}^{2}+\eta(2+\xi)K_{33}^{2}&=0, \label{b3} \\
        -(1-\xi)K_{12}^{1}-\xi\eta K_{13}^{1}+\eta(\xi-2)K_{23}^{2}+(1-\xi)K_{33}^{2}+\eta K_{33}^{3}&=0, \label{d2}
    \end{align}
    \begin{align}
        -\xi\eta K_{12}^{1}+(1+\xi)K_{13}^{1}+\xi\eta K_{22}^{2}-2\xi K_{23}^{2}-\xi\eta K_{33}^{2}&=0, \label{e1} \\
        -(1-\xi)K_{12}^{1}+\xi\eta K_{13}^{1}+\xi\eta K_{23}^{2}-2\xi K_{33}^{2}-\xi\eta K_{33}^{3}&=0. \label{f1}
    \end{align}
    Note that the equations \eqref{a2}-\eqref{f3} and the equations \eqref{a1}-\eqref{f1} are independent. 
	First we consider the system \eqref{a2}-\eqref{f3} and determine 
	$K_{11}^{1}, K_{22}^1, K_{23}^{1}$ and $K_{33}^{1}$. From \eqref{a3} and \eqref{c2} we get
    \begin{align}
        -2\xi(\eta K_{22}^{1}+2K_{23}^{1}+\eta K_{33}^{1})=0. \label{eq1}
    \end{align}
    Since $\xi \neq 0$, from \eqref{e2} and \eqref{e3} 
	we have $K_{23}^{1}=\frac{\xi\eta}{1+\xi}K_{22}^{1}$ and $K_{33}^{1}=\frac{1-\xi}{1+\xi}K_{22}^{1}$. 
	Substituting $K_{23}^1$ and $K_{33}^1$ into the equation \eqref{eq1}, we obtain
    \begin{align*}
        -4\xi \eta K_{22}^1=0.
    \end{align*}
    This implies that  $\xi=0$, $\eta=0$ or $K_{22}^1=0$.
    \begin{enumerate}

        \item The case $\xi=0$: In this case $\mathfrak{g}=\mathfrak{g}(0,\eta)$ is the Lie algebra 
		of $\mathbb{H}^3(-1)$ (see Example \ref{Example2.2}).
		From \eqref{e2}-\eqref{f3}, we obtain $K_{23}^1=0$ and $K_{22}^1=K_{33}^1$. 
		Since $K_{23}^1=0$, we obtain $K_{11}^1=2K_{22}^1=2K_{33}^1$ 
		from \eqref{a2} and \eqref{c3}. In this case, the equations 
		\eqref{a3} and \eqref{c2} hold for any $\eta$. Thus $K_{11}^1, K_{22}^1,$ and $K_{33}^1$ are parametrized as
        \begin{align}
            K_{11}^1=2K_{22}^1=2K_{33}^1=\alpha, \quad \alpha \in \mathbb{R}. \label{xizero}
        \end{align}
		On the other hand, from \eqref{a1}--\eqref{f1}, we deduce that 
		\[
		K_{12}^{1}=K_{13}^{1}=K_{22}^{2}=K_{23}^{2}=K_{33}^{2}=K_{33}^{3}=0.
		\]

		\medskip
		

		



\medskip

The skewness operator $K^{(\alpha)}$ determined by \eqref{xizero} has the 
form $K^{(\alpha)}=\alpha\,K^{(1)}$. 
Thus 
$\nabla^{g}-K^{(\alpha)}/2$ is the $\alpha$-connection of $\nabla^{g}-K^{(1)}/2$.

        \item The case $\eta=0$ and $\xi\neq 0$: 
		In this case $\mathfrak{g}=\mathfrak{g}(\xi,0)$ is the Lie algebra 
		of $\mathbb{H}^2(-4)\times\mathbb{R}$ (see Example \ref{Example2.3}). 
        From \eqref{e2}, we obtain $K_{23}^1=0$ and hence 
		the equations \eqref{a3}, 
		\eqref{c2} and \eqref{f3} are automatically satisfied. 
		Moreover the equations \eqref{a2} and \eqref{c3} are reduced to
        \begin{align*}
            K_{11}^1-2K_{22}^1=0, \quad (1-\xi)(K_{11}^1-2K_{33}^1)=0.
        \end{align*}
        If $\xi \neq 1$, then we have 
		$K_{11}^1=2K_{22}^1=2K_{33}^1$.
	 In addition we get 
	 $\xi(K_{22}^{1}+K_{33}^{1})=0$ from 
	 \eqref{e3}. Since we assumed that $\xi\not=0$, we have 
	 $K_{22}^{1}+K_{33}^{1}=0$. Hence 
	 $K_{11}^{1}=K_{22}^{1}=K_{33}^{1}=0$. 
	 
	 Next, from \eqref{b2}, \eqref{d3} and 
	 \eqref{b3}, we get 
	\[
	K_{22}^{2}=2K_{12}^{1},
	\quad 
	K_{33}^{3}=2K_{13}^{1},
	\quad 
	K_{23}^{2}=K_{13}^{1}.
	\]
	The equations \eqref{d2} and \eqref{e1} imply that 
	$K_{23}^{2}=K_{33}^{2}=0$. Thus we conclude that 
	$K=0$. This is a contradiction. Hence $\xi=1$.

%

Now let us consider the Lie algebra $\mathfrak{g}(1,0)$. 
Then from \eqref{e3}, \eqref{a1} and \eqref{b2}, we get
\[
K_{33}^{1}=K_{12}^{1}=K_{22}^{2}=0.
\]
Next, equations \eqref{b3} and \eqref{e1} implies that 
$K_{13}^{1}=K_{23}^{2}$.
Finally \eqref{f1} implies $K_{33}^{2}=0$. 
Henceforth, $K$ has nontrivial components
        \begin{align}
            K_{11}^1=2K_{22}^1=\alpha, \quad 
			K_{13}^{1}=K_{23}^{2}=\beta, 
			\quad 
			K_{33}^{3}=\gamma,
			\quad 
			\alpha,\beta,\gamma\in\mathbb{R}. \label{xione}
        \end{align}

\medskip
		
\medskip

        \item The case $K_{22}^1=0$:
        From \eqref{e2} and \eqref{e3}, $K_{23}^1=K_{33}^1=0$. Then we obtain $K_{11}^1=0$ from \eqref{a2}.
        
    \end{enumerate}
    
Next, we consider the system \eqref{a1}-\eqref{f1} and determine 
$K_{12}^1$, $K_{13}^1$, $K_{22}^2$, $K_{23}^2$, 
$K_{33}^2$ and $K_{33}^3$. 
Set the column vectors $\Vec{k}=(K_{12}^{1},K_{13}^{1},K_{22}^{2},
K_{23}^{2},K_{33}^{2},K_{33}^{3})^T$ and 
$\Vec{k}_2=(K_{12}^{1},K_{13}^{1})^T$. 
Then the system \eqref{b2}-\eqref{f1} [resp. \eqref{a1}-\eqref{c1}] 
is expressed as $\mathcal{A}\Vec{k}_1=\Vec{0}$ 
[resp. $\mathcal{B}\Vec{k}_2=\Vec{0}$]. 
Here the coefficient matrices $\mathcal{A}$ and 
$\mathcal{B}$ are given by
        \begin{align}
            \mathcal{A}&:=
                \begin{pmatrix}
                    -2(1+\xi) & 0 & 1+\xi & \eta(3+\xi) & 0 & 0 \\
                    0 & -2(1-\xi) & 0 & 0 & -\eta(3-\xi) & 1-\xi \\
                    -\xi\eta & -(1+\xi) & -\eta & 1+\xi & \eta(2+\xi) & 0 \\
                    -(1-\xi) & -\xi\eta & 0 & \eta(\xi-2) & 1-\xi & \eta \\
                    -\xi\eta & 1+\xi & \xi\eta & -2\xi & -\xi\eta & 0 \\
                    -(1-\xi) & \xi\eta & 0 & \xi\eta & -2\xi & -\xi\eta
                \end{pmatrix}
                ,\\
            \mathcal{B}&:=
                \begin{pmatrix}
                    3(1+\xi) & \eta(1+3\xi) \\
                    \eta(3\xi-1) & 3(1-\xi) 
                \end{pmatrix}, 
        \end{align}
		respectively. 
		It is easy to see that if $\det \mathcal{A}\neq 0$ [resp. $\det \mathcal{B}\neq 0$], 
		then the coefficients $K_{12}^1, K_{13}^1, K_{22}^2, K_{23}^2, K_{33}^2,$ and $K_{33}^3$ 
		[resp. $K_{12}^1$ and $K_{13}^1$] are equal to $0$. Since we assumed that $K\not=0$,  
		we deduce that $\det \mathcal{A}=0$. Here we compute
        \begin{align}
            0=\det \mathcal{A} = (1-\xi)(1+\xi)(1+\eta^2)(1-\xi^2-\xi^2\eta^2)(1-\xi^2+9\eta^2-\xi^2\eta^2),
        \end{align}
        and from $\xi,\eta\ge 0$ we obtain 
        \begin{align}
            1-\xi=0, \quad 1-\xi^2-\xi^2\eta^2=0 \quad\textrm{or}\quad 1-\xi^2+9\eta^2-\xi^2\eta^2=0. \label{condition}
        \end{align}
         Since $\det \mathcal{B} = 9-9\xi^2+\eta^2-9\xi^2\eta^2$, in all cases $\det \mathcal{B}=0$ if and only if $\eta=0$. 
		 

Let us assume that $\eta\not=0$. 
Then, since $K_{12}^1=K_{13}^1=0$, the equations \eqref{b2}, \eqref{b3} and \eqref{e1} are reduced to
\[
\left(
\begin{array}{ccc}
1+\xi & \eta(3+\xi) & 0\\
-\eta & 1+\xi & \eta(2+\xi)\\
\xi\eta & -2\xi & -\xi\eta
\end{array}
\right)
\left(
\begin{array}{c}
K_{22}^{2}
\\
K_{23}^{2}
\\
K_{33}^{2}
\end{array}
\right)
=\left(
\begin{array}{c}
0
\\
0
\\
0
\end{array}
\right).
\]
The determinant of the coefficient matrix is $\xi\eta(1+\xi)(3+\xi)(1+\eta^2)$ 
which is non-zero if and only if $\xi>0$  because $\xi\ge 0$ and $\eta\neq 0$. 
Thus if $\xi>0$, we have $K_{22}^2=K_{23}^2=K_{33}^2=0$. In this case, 
from \eqref{d2} we have $K_{33}^3=0$. Hence $K=0$. This contradicts to the assumption 
$K\not=0$. Thus we have $\xi=0$. However, $\xi=0$ contradicts the equations \eqref{condition}. 
Thus we conclude that $\eta=0$.

%
%

 
Now let us consider the case $K_{22}^{1}=0$ and $\eta=0$. 
Considering the equations \eqref{condition} we obtain $\xi=1$. Thus the Lie algebra 
$\mathfrak{g}(1,0)$ is the Lie algebra of $\mathbb{H}^{2}(-4)\times\mathbb{R}$ 
(see Example \ref{Example2.3}).

From \eqref{a1}, \eqref{b2} and \eqref{f1} we get $K_{12}^1=K_{22}^2=K_{33}^2=0$. Moreover, from \eqref{b3} we obtain $K_{13}^1=K_{23}^2$, and the other equations hold. Thus $K$ has non-trivial components

            \begin{align}
                K_{13}^1=K_{23}^2=\beta, 
				\quad K_{33}^{3}=\gamma,
				\quad \beta,\gamma\in\mathbb{R}. \label{xione*}
            \end{align}
One can see that the skewness operator $K$ given by \eqref{xione*} 
coincides with the one determined by \eqref{xione} with $\alpha=0$. 

We conclude that the Lie group 
$G$ admits non-trivial conjugate symmetric left invariant statistical structures 
when and only when its Lie algebra is 
$\mathfrak{g}(0,\eta)$ or $\mathfrak{g}(1,0)$.

\end{proof}


\section{Concluding remarks}\label{sec:9}
To close this article, we propose some problems:

\subsection{Totally geodesic maps}\label{sec:9.1}
Let $(M,\nabla)$ and 
$(N,{}^{N}\!\nabla)$ be affine manifolds. 
For a smooth map $\psi:M\to N$, 
we denote by $\psi^{*}TN$ the 
pull-back vector bundle of $TM$ by $\psi$.
The linear connection ${}^N\!\nabla$ induces a 
linear connection on $\psi^{*}TN$. The induced connection 
is denoted by $\nabla^{\psi}$. 
The \emph{second fundamental form} 
$\nabla\mathrm{d}\psi$ of $\psi$ is defined by
\[
(\nabla\mathrm{d}\psi)(Y;X)=
\nabla^{\psi}_{X}\mathrm{d}\psi(Y)-
\mathrm{d}\psi(\nabla_{X}Y),
\quad 
X,Y\in\mathfrak{X}(M).
\]
A smooth map $\psi$ is said to be a 
\emph{totally geodesic map} if its second 
fundamental form vanishes (see Vilms \cite{Vilms}).
 
\subsection{Harmonic maps on statistical manifolds}\label{sec:9.2}
In Riemannian geometry, harmonic map theory 
is a one of the central topic of geometric analysis. 
Moreover harmonic map theory is closely related to 
Theoretical 
Physics. Indeed 
harmonic maps of Riemann surfaces into Riemannian symmetric spaces 
are nothing but the \emph{Euclidean non-linear sigma models} 
in Theoretical Physics (see Zakrzewski's book \cite{Zak}).
Here we recall the notion of harmonic map.

Let $(M,g)$ and $(N,h)$ be Riemannian manifolds. 
Assume that $M$ is oriented by a volume element 
$\mathrm{d}v_g$. For a smooth map 
$\psi:M\to N$, its \emph{Dirichlet energy} over a compact region 
$\mathcal{D}\subset M$ is defined by
\[
E(\psi;\mathcal{D})=\int_{\mathcal D}\frac{1}{2}|\!|
\mathrm{d}\psi|\!|^2\,
\mathrm{d}v_g.
\]
A smooth map 
$\psi$ is said to be a \emph{harmonic map} if it is a critical 
point of the Dirichlet energy over any compact region of $M$. 
The Euler-Lagrange equation of the harmonic map is 
\[
\tau(\psi)=\mathrm{tr}_{g}(\nabla\mathrm{d}\psi)=0.
\]
Here $\nabla\mathrm{d}\psi$ is the second fundamental form of 
$\psi$. More precisely we regard $\psi$ as a map 
$\psi:(M,\nabla^{g})\to (N,\nabla^h)$. 
The section $\tau(\psi)$ of $\psi^{*}TN$ is called 
the \emph{tension field} of $\psi$.

Some mapping spaces do not contain any harmonic map.
For instance the space $C^{\infty}_1(\mathbb{T}^2,\mathbb{S}^2)$ of smooth maps 
from a $2$-torus into the $2$-sphere of mapping degree $1$ does not contain 
harmonic maps. On such a space, to look for good representatives, 
the study of bienergy was suggested. 
For a smooth map $\psi:(M,g)\to (N,h)$, the 
\emph{bienergy} of $\psi$ is introduced as
\[
E_{2}(\psi;\mathcal{D})=\int_{\mathcal D}
\frac{1}{2}|\!|\tau(\psi)|\!|^2\,\mathrm{d}v_g.
\]
A smooth map $\psi$ said to be a \emph{biharmonic map} if it is a critical 
point of the bienergy over any compact region of $M$.

Park \cite{Park} studied inner automorphisms 
of compact connected semi-simple Lie groups 
(equipped with a bi-invariant Riemannian metric) which are 
harmonic maps. Biharmonic homomorphims are studied in 
\cite{BoBo,BoOu}.

We may expect to study harmonic automorphims 
of statistical Lie groups. How to introduce the notion 
of harmonicity for automorphisms on statistical Lie groups ?

As we mentioned before, 
statistical manifolds are regarded as 
a Riemannian manifold equipped with a pair 
$(\nabla,\nabla^*)$ of linear connections satisfying 
\eqref{eq:conjugate}.
From this viewpint we concentrate our 
attention to smooth maps 
$\psi:(M,g,\nabla)\to (M,g,\nabla^{*})$. 
We equipe the connection ${}^{*}\nabla^{\psi}$ 
on $\psi^{*}TM$ induced from the conjugate 
connection $\nabla^{*}$ of $\nabla$. Then we define the \emph{statistical tension field} 
$\tau(\psi)$ by 
\[
\tau(\psi)=\mathrm{tr}_{g}(\nabla\mathrm{d}\psi)
=\sum_{i=1}^{m}
(\nabla\mathrm{d}\psi)(e_i,e_i),\quad m=\dim M,
\] 
where $\{e_1,e_2,\dots,e_m\}$ is a 
local orthonormal frame field of $M$ relative to the 
Riemannian metric $g$. 
Note that the statistical tension field 
is rewritten as 
\[
\tau(\psi)=\sum_{i=1}^{m}
\left(
{}^{*}\nabla^{\psi}_{e_i}\mathrm{d}\psi(e_i)-\mathrm{d}\psi(\nabla_{e_i}e_i)
\right).
\]
Here we propose the following definition.

\begin{Definition}
{\rm Let $M=(M,g,\nabla)$ be a statistical manifold. 
A smooth map $\psi:M\to M$ is said to be a 
\emph{statistically harmonic map} if its 
statistical tension field vanishes.
}
\end{Definition} 
 
\begin{Problem}
Study statistically harmonic automorphisms of statistical Lie groups.
\end{Problem} 
 
Let us now explain that the statistically 
harmonic map theory is markedly different from the harmonic map theory in 
Riemannian geometry. Indeed, the identity map $\mathrm{id}$ of a Riemannian manifold 
is harmonic. The stability of the identity map (as a harmonic map) 
has been paid much attention of differential geometers. 
The stability of identity maps of compact semi-simple Riemannian symmetric spaces 
has similarity to the Yang-Mills stability 
over compact semi-simple Riemannian symmetric spaces. 

However, in the case of statistical manifolds, 
one can see that the statistical tension field of 
the identity map of a statistical manifold 
is 
\[
\tau(\mathrm{id})=-2\mathrm{tr}_{g}K,
\]
where $K$ is the skewness operator. This formula implies 
that $M$ satisfies the apolarity if and only if $\mathrm{id}$ is statistically 
harmonic. On a statistical manifold satisfying $\tau(\mathrm{id})\not=0$, 
one may consider several alternative conditions.
For instance we may study 
statistical manifolds with \emph{parallel} $\tau(\mathrm{id})$, 
\textit{i.e.}, $\nabla \tau(\mathrm{id})=0$ (or 
$\nabla\tau(\mathrm{id})=\nabla^{*}\tau(\mathrm{id})=0$). 

It should be remarked that when $M$ is a Hessian manifold, 
then its Hessian curvature tensor field $H$ is defined by
\[
H(X,Y)Z=\frac{1}{2}(\nabla_{X}K)(Y,Z).
\]
Another candidate is the biharmonicity. 
For a smooth map $\psi:M\to M$, one may consider 
the bienergy like as Riemannian geometry by:
\[
E_{2}(\psi;\mathcal{D})=\int_{\mathcal D}
\frac{1}{2}|\!|\tau(\psi)|\!|^2\,\mathrm{d}v_g.
\]
\begin{Problem}
Study statistically biharmonic automorphisms of statistical Lie groups.
\end{Problem} 

\subsection{Homogeneous holomorphically statistical manifolds}
In Section \ref{sec:2.3} we discussed homogeneous Sasakian statistical manifolds. 
On the other hand in even-dimensional geometry, the notion of 
holomorphic statistical manifold was introduced in the following manner
(see \textit{eg.} \cite{Mirja,Mirja2}).

\begin{Definition}{\rm 
Let $(M,g,J)$ be a K{\"a}hler manifold. 
A torsion free linear connection $\nabla$ is compatible to $(g,J)$ is 
said to be \emph{compatible} if 
$C:=\nabla g$ is totally symmetric and 
the K{\"a}hler form 
$\varOmega=g(\cdot,J)$ is parallel with respect to $\nabla$. 
Such a triplet $(g,J,\nabla)$ is called a 
\emph{holomorphic statistical structure}. 
A manifold $M$ equipped with a holomorphic statistical structure 
is called a \emph{holomorphic statistical manifold}.
}
\end{Definition}
According this definition, 
a holomorphic statistical manifold 
$(M,g,J,\nabla)$ is said to be 
\emph{homogeneous} if there exits a 
Lie group $G$ acting transitively on $M$ so that 
the action is isometric with respecto to $g$, 
holomorphic with respect to 
$J$ and affine with respect to $\nabla$.

\begin{Problem}
Construct explicit examples 
of homogeneous holomorphic statistical manifolds.
\end{Problem} 
Let $(M,g,J)$ be a K{\"a}hler manifold and $\nabla$ a 
linear connection on $M$. Then one can see 
that $(g,J,\nabla)$ is a holomorphic statistical structure if and only if 
$K=-2(\nabla-\nabla^g)$ satisfies 
\[
K(X,Y)=K(Y,X),
\quad 
g(K(X,Y),Z)=g(Y,K(X,Z)),
\quad 
K(X,JY)+JK(X,Y)=0.
\]
This is rewritten as
\[
C(X,Y,Z)\>\mbox{is totally symmetric and}
\quad 
C(X,JY,Z)=C(X,Y,JZ).
\]
Then by virtue of Sekigawa-Kirichenko theorem,
we obtain
\begin{Corollary}
Let $(M,g,J,C)$ be a K{\"a}hler manifold equipped 
with a cubic form $C$ satisfying 
$C(X,JY,Z)=C(X,Y,JZ)$. 
Set $\nabla=\nabla^g-K/2$, where 
$C(X,Y,Z)=g(K(X,Y),Z)$. 
If $(M,g,J,\nabla)$ is is homogeneous 
holomorphic statistical manifold, then 
there exists a homogeneous Riemannian structure 
$S$ satisfying 
$\tilde{\nabla}C=0$ and $\tilde{\nabla}J=0$ with 
respect to 
the canonical connection $\tilde{\nabla}=\nabla+S$.

Conversely, let $(M,g,J,\nabla)$ be  
a holomorphic statistical manifold with 
cubic form $C$ admitting a homogeneous Riemannian
structure $S$ satisfying 
$\tilde{\nabla}C=0$ and $\tilde{\nabla}J=0$ with 
respect to 
the canonical connection $\tilde{\nabla}=\nabla+S$, then 
$(M,g,J,\nabla)$ is locally homogeneous.
\end{Corollary}

\subsection{Geodesics on homogeneous statistical manifolds}
Up to now there is no notion of 
``symmetric space" for statistical manifolds.
Assume that $M=G/H$ be a reductive homogeneous 
statistical manifold and $G$ acts effectively on $M$. 
From algebraic point of view, 
we may propose the following definition.
\begin{Definition}
Let $M=G/H$ be homogeneous statistical manifold. 
Then $M$ is said to be a 
\emph{statistical symmetric space} if it admits a 
reductive decomposition 
$\mathfrak{g}=\mathfrak{h}
+\mathfrak{m}$ satisfying 
\[
[\mathfrak{m},\mathfrak{m}]\subset 
\mathfrak{h}.
\] 
\end{Definition}
In Riemannian geometry, 
local symmetry is characterized as 
the parallelism of the Riemannian curvature 
with respect to the Levi-Civita connection. 
Moreover local symmetry is equivalent to 
that all the local geodesic 
symmetries are isometric.

\begin{Problem}
Clarify geometric meaning 
of the statistical symmetry 
in terms of curvature or geodesic.
\end{Problem}

In information geometry, 
geodesics with respect to $\nabla$ and 
$\nabla^*$ play important roles. 
For instance the generalized 
Pythagorean-theorem for dually flat 
statistical manifolds is 
well known. 

On the other hand, from dynamical system viewpoint 
(or analytical mechanics), 
integrability of geodesic flows is 
an important issue.

In Riemannian geometry, 
\emph{geodesic orbit spaces} have been 
paid 
much attention of homogeneous geometers.

A reductive homogeneous Riemannian space 
$M=G/H$ is said to be a 
\emph{geodesic orbit space} 
(Riemannian g.~o.~space) if 
its all the geodesic starting at the origin 
$o$ are homogeneous, that is, all those geodesics 
have the form 
$\exp(tX)\cdot o$, where $X\in\mathfrak{g}$.
The geodesic flows of a Riemannian g.~o.~space is
integrable. 
Riemannian symmetric spaces are Riemannian 
g.~o.~spaces.

\begin{Problem}
Classify homogeneous 
statistical manifolds 
whose $\nabla$-geodesics and 
$\nabla^*$-geodesics are homogeneous.
\end{Problem}

\bibliographystyle{plain}
\def\cprime{$'$}

\end{document}